\documentclass[11pt]{article}
\usepackage{amssymb,amsthm}
\usepackage{amsmath,latexsym,epsf,mathabx,underscore}
\usepackage[sort&compress,numbers]{natbib} 
\usepackage{changepage}
\usepackage[all]{xy}
\usepackage{hyperref}
\numberwithin{equation}{section}
\usepackage{color}
\begin{document}

\newcommand{\nc}{\newcommand}
\def\PP#1#2#3{{\mathrm{Pres}}^{#1}_{#2}{#3}\setcounter{equation}{0}}
\def\mr#1{{{\mathrm{#1}}}\setcounter{equation}{0}}
\def\mc#1{{{\mathcal{#1}}}\setcounter{equation}{0}}
\def\mb#1{{{\mathbb{#1}}}\setcounter{equation}{0}}
\def\Mcc{\mc{C}}
\def\Mbe{\mb{E}}
\def\Mcp{\mc{P}}
\def\Mcg{\mc{G}}
\def\fbzh{\mc{C}(-,\mc{P}(\xi))\text{-exact}}
\def\extri{(\mc{C},\mb{E},\mathfrak{s})}
\def\Gproj{\xi\text{-}\mc{G}\text{projective}}
\def\Ginj{\xi\text{-}\mc{G}\text{injective}}
\def\GP{\mc{G}\mc{P}(\xi)}
\def\GI{\mc{G}\mc{I}(\xi)}
\def\P{\mc{P}(\xi)}
\def\Gpd{\xi\text{-}\mc{G}\text{pd}}
\def\Extri{\mb{E}\text{-triangle}}
\def\SGP{\mc{SGP}(\xi)}
\def\nSGP{\text{$n$}\text{-}\mc{SGP}(\xi)}
\def\mSGP{\text{$m$}\text{-}\mc{SGP}(\xi)}
\def\ext{\xi \text{xt}_{\xi}}
\def\extgp{\xi \text{xt}_{\mc{GP}(\xi)}}
\def\extgi{\xi \text{xt}_{\mc{GI}(\xi)}}
\def\gext{\mc{G}\xi\text{xt}_{\xi}}
\def\nGpd{n\text{-}\xi\text{-}\mc{SG}\text{pd}}
\def\ext{\xi \text{xt}_{\xi}}
\newtheorem{defn}{\bf Definition}[section]
\newtheorem{cor}[defn]{\bf Corollary}
\newtheorem{prop}[defn]{\bf Proposition}
\newtheorem{thm}[defn]{\bf Theorem}
\newtheorem{lem}[defn]{\bf Lemma}
\newtheorem{rem}[defn]{\bf Remark}
\newtheorem{exam}[defn]{\bf Example}
\newtheorem{fact}[defn]{\bf Fact}
\newtheorem{cond}[defn]{\bf Condition}
\def\Pf#1{{\noindent\bf Proof}.\setcounter{equation}{0}}
\def\>#1{{ $\Rightarrow$ }\setcounter{equation}{0}}
\def\<>#1{{ $\Leftrightarrow$ }\setcounter{equation}{0}}
\def\bskip#1{{ \vskip 20pt }\setcounter{equation}{0}}
\def\sskip#1{{ \vskip 5pt }\setcounter{equation}{0}}
\def\mskip#1{{ \vskip 10pt }\setcounter{equation}{0}}
\def\bg#1{\begin{#1}\setcounter{equation}{0}}
\def\ed#1{\end{#1}\setcounter{equation}{0}}




\title{\bf Gorenstein Objects in Extriangulated Categories}

\smallskip
\author{\small Zhenggang He\\
\small E-mail:~zhenggang_he@163.com\\
\small Institute of Mathematics, School of Mathematics Sciences\\
\small Nanjing Normal University, Nanjing \rm210023 China}
\date{}
\maketitle
\baselineskip 15pt
%
%
%
\vskip 10pt%
\noindent {\bf Abstract}: This thesis mainly studies  the relative Gorenstein objects in the extriangulated category $\Mcc$ with a proper class $\xi$ and the related properties of these objects.
We define the notion of the  $\Gproj$ resolution, and study  the relation between $\xi$-projective resolution and $\xi$-$\mathcal{G}$projective resolution for any object $A$ in $\mathcal{C}$.
What is more, we define a particular $\xi$-Gorenstein projective object in $\Mcc$ which called $\xi$-$n$-strongly $\mathcal{G}$projective object and  study the relation between $\xi$-$m$-strongly $\mathcal{G}$projective object and $\xi$-$n$-strongly $\mathcal{G}$projective object whenever $m\neq n$, and give some equivalent characterizations of $\xi$-$n$-strongly $\mathcal{G}$projective objects.

\mskip\





%
%
\vskip 30pt

\section{Introduction}
Relative homological algebra has been formulated by Hochschild \cite{GH} in categories of modules and 
later Heller, Butler and Horrocks in general categories with a relative abelian structure.
Its main theory includes the extension for a class of objects, and it is natural to consider 
the extension for a class consisting of some triangles in triangulated categories. Based on this,
 Beligianuls \cite{BEL} developed a the homology algebra in triangulated categories which parallels t
 he homological algebra in an exact category in the sense of Quillen. By specifying a class of triangles
  $\xi$, which is called a proper class of triangles, he introduced $\xi$-projective objects,
  $\xi$-projective and $\xi$-global dimensions and their duals.

Auslander and Bridger \cite{AB} introduced a special module with G-dimension zero, which generalized 
the class of finitely generated projective modules over a commutative Noetherian ring. Whereafter, 
Enochs and Jenda \cite{EO1} introduced Gorenstein projective modules over any ring which generalized 
the notion of G-dimension zero modules, and dually they defined Gorenstein injective modules. 
Beligiannis \cite{BEL1} defined $\mathcal{X}$-Gorenstein object in an additive category $\Mcc$ for 
a contravariantly finite subcategory  $\mathcal{X}$ of $\Mcc$ such that any $\mathcal{X}$-epic has
 kernel in $\Mcc$ as a natural generalization of modules of G-dimension zero. In order to extend
  the theory, Asadollahi and Salarian \cite{JS} introduced and studied $\xi$-Gprojective 
  and $\xi$-Ginjective objects, and then $\xi$-Gprojective and $\xi$-Ginjective dimensions
   of objects in a triangulated category with a proper class $\xi$.

Recently, Nakaoka and Palu \cite{HY} introduced an extriangulated category which is extracting 
properties on triangulated categories and exact categories. The class of extriangulated categories
 contains triangulated categories and exact categories as examples. There have been many further 
 researches on extriangulated categories, see \cite{CZZ, YH, ZT,ZZ} etc. Hu, Zhang and Zhou 
 \cite{JDP} developed the above mentioned homological algebra in extriangulated categories. 
 They define a notion of a proper class $\xi$ of $\mathbb{E}$-triangles.  Based on it, they introduced
  the $\xi$-projective objects, $\Gproj$ objects and their duals. Furthermore,  Hu, Zhang and 
  Zhou \cite{JZP} discussed Gorenstein homological dimensions for extriangulated categories and 
  gave some characterizations of $\Gproj$ dimension by using derived functors on $\Mcc$.

Bennis and Mahdou \cite{BN, BM} introduced the notion of strongly Gorenstein projective 
modules and $n$-strongly Gorenstein projective modules. They also gave some equivalent 
characterizations of those modules in terms of the vanishing of some homological groups.
 Yang and Liu \cite{YL} proved that a module $M$ is strongly Gorenstein projective if and 
 only if so is $M\oplus H$ for any projective module $H$. Based on the results mentioned above, 
 Zhao and Huang \cite{ZH} studied the homological behavior of $n$-strongly Gorenstein projective,
  and investigate the relation between $m$-strongly Gorenstein projective modules and $n$-strongly
   Gorenstein projective modules whenever $m\neq n$.

This paper is organized as follows. In section 2, we recall some basic definitions and properties
 which will be of value in later proofs for extriangulated categories. In section 3, we recall 
 some basic definitions and properties of $\xi$-projective and $\xi$-$\mathcal{G}$projective 
 object in an extriangulated category and then we prove that a object has a $\fbzh$ $\xi$-projective 
 resolution  if and only if it has a $\Mcc(-,\P)$-exact $\Gproj$ resolution (see Theorem \ref{ZHFJ}).
  Moreover, we obtain some inequalities for $\xi$-$\mathcal{G}$projective dimension in an
   $\mathbb{E}$-triangle (see Theorem \ref{XMYY}). In section 4, we introduce some special
    $\Gproj$ objects in extriangulated categories  which are called  $\xi$-$n$-strongly
     $\mathcal{G}$projective objects for any integer $n\geqslant1$, we get the relation 
     between $\xi$-$m$-strongly $\mathcal{G}$projective objects  and $\xi$-$n$-strongly
      $\mathcal{G}$projective objects whenever $m\neq n$ (see Theorem \ref{STR2}), 
      and give some equivalent characterizations to the $\xi$-$n$-strongly $\mathcal{G}$projective 
      objects (see Theorem \ref{STR1}).

\section{Preliminaries}

\quad~ In this section, we briefly recall some basic definitions of extriangulated categories. Moreover, we study some related properties which will be of value in later proofs.

Throughout this paper, let $\mathcal{C}$ be an additive category and denote the set of morphisms $A\rightarrow B$ in $\mathcal{C}$ by $\mathcal{C}(A,B)$ for some $A,B\in\mathcal{C}$. If $f\in\mathcal{C}(A,B)$, $g\in\mathcal{C}(B,C)$, then we denote the composition of $f$ and $g$ by $g f$.

Now, we introduce the definition of extriangulated categories. For more details, we refer to \cite{YH} and \cite{HY}.

\begin{defn}\citep[Definition 2.1]{HY}
	Suppose that $\mathcal{C}$ is equipped with an additive bifunctor 
	$\mathbb{E}:\mathcal{C}^{op}\times \mathcal{C}\rightarrow \mathbf{Ab}$. 
	For any pair of objects $A, C$ in $\mathcal{C}$, an element $\delta \in\mathbb{E}(C,A)$ is called an $\mathbb{E}$-\emph{extension}. Thus formally, an $\mathbb{E}$-extension is a triplet $(A,\delta,C)$. Since $\mathbb{E}$ is a functor, for any $a\in \mathcal{C}(A, A')$ and $c\in \mathcal{C}(C, C)$, we have $\mathbb{E}$-extensions $\mathbb{E}(C, a)(\delta)\in \mathbb{E}(C, A') $ and $\mathbb{E}(c,A)(\delta)\in \mathbb{E}(C', A).$ We abbreviately denote them by $a_{*}\delta$ and $c^* \delta$ respectively. In this terminology, we have
	$$ \mathbb{E}(c, a)(\delta)=c^*a_*\delta=a_*c^*\delta$$
	in $\mathbb{E}(C', A')$. For any $ A$, $ C\in \mathcal{C}$, the zero
	element $0\in\mathbb{E}(C, A)$ is called the split $\mathbb{E}$-extension.
\end{defn}

\begin{defn}\citep[Definition 2.3]{HY}
	Let $\delta\in \mathbb{E}(C,A)$ and $\delta '\in \mathbb{E}(C',A')$ be any pair of $\mathbb{E}$-extensions. A morphism $(a,c):\delta\rightarrow\delta'$ of $\mathbb{E}$-extensions is a pair of morphism $a\in \mathcal{C}(A,A')$ and $c\in \mathcal{C}(C,C')$ in $\mathcal{C}$ satisfying the equality
	$$ a_*\delta=c^*\delta '.$$
\end{defn}

\begin{defn}\citep[Definition 2.6]{HY}
	Let $\delta=(A,\delta,C)$ and $\delta'=(A',\delta ',C')$ be any pair of $\mathbb{E}$-extensions. Let $C \stackrel{l_C}\longrightarrow C\oplus C' \stackrel{l_{C'} }\longleftarrow C'$ and $A \stackrel{p_A}\longleftarrow A\oplus A'\stackrel{p_{A'}}\longrightarrow A'$ be coproduct and product in $\mathcal{C}$, respectively. We have a natural isomorphism
	$$ \mathbb{E}(C\oplus C',A\oplus A')\simeq\mathbb{E}(C,A) \oplus \mathbb{E}(C,A')\oplus \mathbb{E}(C',A)\oplus \mathbb{E}(C',A')$$
	by the additivity of $\mathbb{E}$.
\end{defn}

	Let $\delta\oplus\delta'\in \mathbb{E}(C\oplus C',A\oplus A')$ be the element corresponding to $(\delta,0,0,\delta')$ through this isomorphism.

\begin{defn}\citep[Definition 2.7]{HY}
	Let $A$, $C\in \mathcal{C}$ be any pair of objects. Two sequences of morphisms $A\stackrel{x}\longrightarrow B\stackrel{y}\longrightarrow C$ and $A\stackrel{x'}\longrightarrow B'\stackrel{y'}\longrightarrow C$ in $\mathcal{C}$ are said to be \emph{equivalent} if there exists an isomorphism $b\in \mathcal{C}(B,B')$ which makes the following diagram commutative.
	$$
	\xymatrix{
		A\ar@{=}[d]\ar[r]^x &B\ar[d]^b_{\simeq}\ar[r]^y &C\ar@{=}[d]\\
		A\ar[r]^{x'} &B'\ar[r]^{y'} &C}
	$$
\end{defn}	

We denote the equivalence class of $A\stackrel{x}\longrightarrow B\stackrel{y}\longrightarrow C$ by [$A\stackrel{x}\longrightarrow B\stackrel{y}\longrightarrow C$].

\begin{defn}\citep[Definition 2.8]{HY}
	(1) For any $A$, $C\in \mathcal{C}$, we denote as
	$$0=[A \stackrel{\begin{tiny}\begin{bmatrix}
		1 \\
		0
		\end{bmatrix}\end{tiny}}{\longrightarrow} A \oplus C \stackrel{\begin{tiny}\begin{bmatrix}
		0&1
		\end{bmatrix}\end{tiny}}{\longrightarrow} C].$$
	(2) For any $[A\stackrel{x}\longrightarrow B\stackrel{y}\longrightarrow C]$ and $[A'\stackrel{x'}\longrightarrow B'\stackrel{y'}\longrightarrow C']$, we denote as
	$$
	[A\stackrel{x}\longrightarrow B\stackrel{y}\longrightarrow C]\oplus[A'\stackrel{x'}\longrightarrow B'\stackrel{y'}\longrightarrow C']=[A\oplus A'\stackrel{x\oplus x'}\longrightarrow B\oplus B'\stackrel{y\oplus y'}\longrightarrow C\oplus C'].
	$$
\end{defn}

\begin{defn}\citep[Definition 2.9]{HY}
	Let $\mathfrak{s}$ be a correspondence which associates an equivalence class $\mathfrak{s}(\delta)=[A\stackrel{x}{\longrightarrow}B\stackrel{y}{\longrightarrow}C]$ to any $\mathbb{E}$-extension $\delta\in\mathbb{E}(C,A)$ . This $\mathfrak{s}$ is called a  \emph{realization} of $\mathbb{E}$, if for any morphism $(a,c):\delta\rightarrow\delta'$ with $\mathfrak{s}(\delta)=[A\stackrel{x}\longrightarrow B\stackrel{y}\longrightarrow C]$ and $\mathfrak{s}(\delta')=[A'\stackrel{x'}\longrightarrow B'\stackrel{y'}\longrightarrow C']$, there exists $b\in \mathcal{C}$ which makes the following  diagram commutative
	$$\xymatrix{
		& A \ar[d]_-{a} \ar[r]^-{x} & B  \ar[r]^{y}\ar[d]_-{b} & C \ar[d]_-{c}    \\
		&A'\ar[r]^-{x'} & B' \ar[r]^-{y'} & C'.    }
	$$
	In the above situation, we say that the triplet $(a,b,c)$ realizes $(a,b)$.
\end{defn}

\begin{defn}\citep[Definition 2.10]{HY}
	Let $\mathcal{C},\mathbb{E}$ be as above. A realization $\mathfrak{s}$ of $\mathbb{E}$ is said to be {\em additive} if it satisfies the following conditions.
	
	(a) For any $A,~C\in\mathcal{C}$, the split $\mathbb{E}$-extension $0\in\mathbb{E}(C,A)$ satisfies $\mathfrak{s}(0)=0$.
	
	(b) $\mathfrak{s}(\delta\oplus\delta')=\mathfrak{s}(\delta)\oplus\mathfrak{s}(\delta')$ for any pair of $\mathbb{E}$-extensions $\delta$ and $\delta'$.
	
\end{defn}

\begin{defn}\citep[Definition 2.12]{HY}
	A triplet $(\mathcal{C}, \mathbb{E},\mathfrak{s})$ is called an \emph{ extriangulated category} if it satisfies the following conditions. \\
	$\rm(ET1)$ $\mathbb{E}$: $\mathcal{C}^{op}\times\mathcal{C}\rightarrow \mathbf{Ab}$ is a biadditive functor.\\
	$\rm(ET2)$ $\mathfrak{s}$ is an additive realization of $\mathbb{E}$.\\
	$\rm(ET3)$ Let $\delta\in\mathbb{E}(C,A)$ and $\delta'\in\mathbb{E}(C',A')$ be any pair of $\mathbb{E}$-extensions, realized as
	$$\mathfrak{s}(\delta)=[A\stackrel{x}{\longrightarrow}B\stackrel{y}{\longrightarrow}C] \quad \text{and} \quad \mathfrak{s}(\delta')=[A'\stackrel{x'}{\longrightarrow}B'\stackrel{y'}{\longrightarrow}C'].$$ For any commutative square	
    $$\xymatrix{
		A \ar[d]_{a} \ar[r]^{x} & B \ar[d]_{b} \ar[r]^{y} & C \\
		A'\ar[r]^{x'} &B'\ar[r]^{y'} & C'}$$
	in $\mathcal{C}$, there exists a morphism $(a,c)$: $\delta\rightarrow\delta'$ which is realized by $(a,b,c)$.\\
	$\rm(ET3)^{op}$ Dual of $\rm(ET3)$.\\
    $\rm(ET4)$ Let $\delta\in \mathbb{E}(D,A)$ and $\delta'\in \mathbb{E}(F,B)$ be $\mathbb{E}$-extensions respectively realized by
	$$A\stackrel{f}{\longrightarrow}B\stackrel{f'}{\longrightarrow}D\quad \text{and }\quad B\stackrel{g}{\longrightarrow}C\stackrel{g'}{\longrightarrow}F.$$
	Then there exist an object $E\in\mathcal{C}$, a commutative diagram
	$$\xymatrix{
		A \ar@{=}[d]\ar[r]^{f} &B\ar[d]_{g} \ar[r]^{f'} & D\ar[d]^{d} \\
		A \ar[r]^{h} & C\ar[d]_{g'} \ar[r]^{h'} & E\ar[d]^{e} \\
		& F\ar@{=}[r] & F   }$$
	in $\mathcal{C}$, and an $\mathbb{E}$-extension $\delta''\in \mathbb{E}(E,A)$ realized by $A\stackrel{h}{\longrightarrow}C\stackrel{h'}{\longrightarrow}E$, which satisfy the following compatibilities.\\
	$(\textrm{i})$ $D\stackrel{d}{\longrightarrow}E\stackrel{e}{\longrightarrow}F$ realizes $f'_*\delta'$,\\
	$(\textrm{ii})$ $d^*\delta''=\delta$,\\
	$(\textrm{iii})$ $f_*\delta''=e^*\delta$.\\
	$\rm(ET4)^{op}$ Dual of $\rm(ET4)$.
\end{defn}
For examples of extriangulated categories, see \citep[Example 2.13]{HY} and \citep[Remark 3.3]{JDP}.

We will use the following terminology.

\begin{defn}\citep[Definition 2.15 and 2.19]{HY}
	Let $(\mathcal{C},\mathbb{E},\mathfrak{s})$ be an extriangulated category.
	
	(1) A sequence $A\stackrel{x}\longrightarrow B\stackrel{y}\longrightarrow C$ is called \emph{conflation} if it realizes some $\mathbb{E}$-extension $\delta\in \mathbb{E}(C,A)$. In this case, $x$ is called an\emph{ inflation} and $y$ is called a \emph{deflation}.
	
	(2) If a conflation $A\stackrel{x}\longrightarrow B\stackrel{y}\longrightarrow C$ realizes $\delta\in \mathbb{E}(C,A)$, we call the pair ($A\stackrel{x}\longrightarrow B\stackrel{y}\longrightarrow C, \delta$) an $\mathbb{E}$-\emph{triangle}, and write it by
	$$\xymatrix{
		A \ar[r]^{x} & B  \ar[r]^{y} & C  \ar@{-->}[r]^{\delta} & }.$$
	We usually don't write this ``$\delta$" if it not used in the argument.
	
	(3) Let $\xymatrix{
		A \ar[r]^{x} & B  \ar[r]^{y} & C  \ar@{-->}[r]^{\delta} & }$ and $\xymatrix{
		A '\ar[r]^{x'} & B'  \ar[r]^{y'} & C'  \ar@{-->}[r]^{\delta'} & }$ be any pair of $\mathbb{E}$-triangles. If a triplet $(a,b,c)$ realizes $(a,c):\delta \rightarrow\delta'$, then we write it as
	
	$$\xymatrix{
		A \ar[d]_{a} \ar[r]^{x} & B  \ar[r]^{y}\ar[d]_{b} & C \ar[d]_{c} \ar@{-->}[r]^{\delta} &  \\
		A'\ar[r]^{x'} & B' \ar[r]^{y'} & C' \ar@{-->}[r]^{\delta'} &    }
	$$
	and call $(a,b,c)$ a\emph{ morphism }of $\mathbb{E}$-triangles.
	
	(4) An $\mathbb{E}$-triangle $\xymatrix{A\ar[r]^{x}&B\ar[r]^{y}&C\ar@{-->}[r]^{\delta}&}$ is called\emph{ split} if $\delta=0$.
\end{defn}

Next, we will introduce some basic properties of extriangulated category.

Assume that $(\mathcal{C},\mathbb{E},\mathfrak{s})$ is an extriangulated category. By Yoneda's Lemma, any $\mathbb{E}$-extension $\delta\in \mathbb{E}(C,A)$ induces natural transformations
$$
\delta_{\sharp}:\mathcal{C}(-,C)\Rightarrow \mathbb{E}(-,A) \text{\;and\;} \delta^{\sharp}:\mathcal{C}(A,-)\Rightarrow \mathbb{E}(C,-).
$$
For any $X\in \mathcal{C}$, $(\delta_{\sharp})_X$ and $\delta^{\sharp}_X$ are defined as follows:

(1) $(\delta_{\sharp})_X:\mathcal{C}(X,C)\Rightarrow\mathbb{E}(X,A); f\mapsto f^*\delta.$

(2) $\delta^{\sharp}_X:\mathcal{C}(A,X)\Rightarrow\mathbb{E}(C,X); g\mapsto g_*\delta.$

\begin{lem}\label{FJ}\citep[Corollary 3.5]{HY}
	Assume that $(\mathcal{C},\mathbb{E},\mathfrak{s})$ satisfies $\rm(ET1)$, $\rm(ET2)$, $\rm(ET3)$ and $\rm(ET3)^{op}$. Let
	$$
	\xymatrix{
		A\ar[r]^x\ar[d]^a&B\ar[r]^y\ar[d]^b&C\ar[d]^c\ar@{-->}[r]^{\delta} & \\
		A'\ar[r]^{x'}&B'\ar[r]^{y'}&C'\ar@{-->}[r]^{\delta'} &
	}
	$$
	be any morphism of $\mathbb{E}$-triangles. Then the following are equivalent.
	
	(1) $a$ factors through $x$.
	
	(2) $a_*\delta=c^*\delta'=0$.
	
	(3) $c$ factors through $y'$.
\end{lem}	
	\noindent In particular, in the case $\delta = \delta'$ and $(a, b, c) =(1_A, 1_B, 1_C )$, we have
	$$
	x \text{\; is a section} \Leftrightarrow \delta \text{\; is split} \Leftrightarrow y \text{\;is a retraction}.
	$$

    \begin{lem}\citep[Corollary 3.12]{HY}
	Let $(\mathcal{C},\mathbb{E},\mathfrak{s})$ be an extriangulated category, and
	$$
	\xymatrix{
		A\ar[r]^x &B\ar[r]^y & C \ar@{-->}[r]^{\delta'} &
	}
	$$
	an $\mathbb{E}$-triangle. Then there are long exact sequences:
	$$
	\xymatrix{
		\mathcal{C}(C,-)\ar[r]^{\mathcal{C}(y,-)} &{\mathcal{C}(B,-)}\ar[r]^{\mathcal{C}(x,-)} &{\mathcal{C}(A,-)}\ar[r]^{\delta^{\sharp}}&{\mathbb{E}(C,-)}\ar[r]^{\mathbb{E}(y,-)} &{\mathbb{E}(B,-)}\ar[r]^{\mathbb{E}(x,-)}&{\mathbb{E}(A,-)}
	};
	$$
	$$
	\xymatrix{
		\mathcal{C}(-,A)\ar[r]^{\mathcal{C}(-,x)} &{\mathcal{C}(-,B)}\ar[r]^{\mathcal{C}(-,y)} &{\mathcal{C}(-,C)}\ar[r]^{\delta_{\sharp}}&{\mathbb{E}(-,A)}\ar[r]^{\mathbb{E}(-,x)} &{\mathbb{E}(-,B)}\ar[r]^{\mathbb{E}(-,y)}&{\mathbb{E}(-,C)}
	}.
	$$
\end{lem}

\begin{lem}\citep[Lemma 3.8]{JDP}
	(1) Let $(\mathcal{C},\mathbb{E},\mathfrak{s})$ be an extriangulated category, $A\stackrel{f}\longrightarrow B\stackrel{f'}\longrightarrow C\stackrel{\delta_f}\dashrightarrow$, $B\stackrel{g}\longrightarrow D\stackrel{g'}\longrightarrow F\stackrel{\delta_g}\dashrightarrow$ and $A\stackrel{h}\longrightarrow D\stackrel{h'}\longrightarrow E\stackrel{\delta_h}\dashrightarrow$ be any triplet of $\mathbb{E}$-triangles satisfying $h=gf$. Then there are morphism $d$ and $e$ in $\mathcal{C}$ which make the diagram
	$$
	\xymatrix{
		A\ar@{=}[d]\ar[r]^{f}&B\ar[d]^{g}\ar[r]^{f'}&C\ar[d]^d\ar@{-->}[r]^{\delta_f}&\\
		A\ar[r]^{h}&D\ar[r]^{h'}\ar[d]^{g'}&E\ar[d]^e \ar@{-->}[r]^{\delta_h}&\\
		&F\ar@{=}[r]\ar@{-->}[d]^{\delta_g}&F\ar@{-->}[d]^{f_*'(\delta_g)}\\
		& &
	}
	$$
	commutative, and satisfying the following compatibilities.
	
	(\romannumeral 1)  $C\stackrel{d}\longrightarrow E\stackrel{e} \longrightarrow F\stackrel{f_*(\delta_g)}\dashrightarrow$ is an $\mathbb{E}$-triangle.
	
	(\romannumeral 2) $d^*(\delta_h)=\delta_f$.
	
	(\romannumeral 3) $e^*(\delta_g)=f_*(\delta_h)$.
	
	(\romannumeral 4) $B \stackrel{\begin{tiny}\begin{bmatrix}
		g \\
		f'
		\end{bmatrix}\end{tiny}}{\longrightarrow} D \oplus C \stackrel{\begin{tiny}\begin{bmatrix}
		h\ -d
		\end{bmatrix}\end{tiny}}{\longrightarrow} E \stackrel{f_*(\delta_h')}\dashrightarrow$ is an $\mathbb{E}$-triangle.
	
	(2) Dual of (1).
\end{lem}

\begin{lem}\label{BH}\citep[Corollary 3.15]{HY}
	Let $(\mathcal{C},\mathbb{E},\mathfrak{s})$ be an extriangulated category. Then the following hold.
	
	(1) Let $C$  be any object, and let $A_1\stackrel{x_1}{\longrightarrow}B_1\stackrel{y_1}{\longrightarrow}C\stackrel{\delta_1}\dashrightarrow$  and $A_2\stackrel{x_2}{\longrightarrow}B_2\stackrel{y_2}{\longrightarrow}C\stackrel{\delta_2}\dashrightarrow$  be any pair of $\mathbb{E}$-triangles. Then there is a commutative diagram in $\mathcal{C}$
	$$
	\xymatrix{
		&{A_2}\ar[d]^{m_2}\ar@{=}[r]&{A_2}\ar[d]^{x_2}\\
		{A_1}\ar@{=}[d]\ar[r]^{m_1}&M\ar[r]^{e_1}\ar[d]^{e_2}&{B_2}\ar[d]^{y_2}\\
		 {A_1}\ar[r]^{x_1}&{B_1}\ar[r]^{y_1}&C}
	$$
	\noindent which satisfies $\mathfrak{s}(y^*_2\delta_1)=[A_1\stackrel{m_1}{\longrightarrow}M\stackrel{e_1}{\longrightarrow}B_2]$ and $\mathfrak{s}(y^*_1\delta_2)=[A_2\stackrel{m_2}{\longrightarrow}M\stackrel{e_2}{\longrightarrow}B_1]$.
	
	(2)  Let $A$  be any object, and let $A\stackrel{x_1}{\longrightarrow}B_1\stackrel{y_1}{\longrightarrow}C_1\stackrel{\delta_1}\dashrightarrow$  and $A\stackrel{x_2}{\longrightarrow}B_2\stackrel{y_2}{\longrightarrow}C_2\stackrel{\delta_2}\dashrightarrow$  be any pair of $\mathbb{E}$-triangles. Then there is a commutative diagram in $\mathcal{C}$
	$$
	\xymatrix{
		A\ar[d]^{x_2}\ar[r]^{x_1}&{B_1}\ar[d]^{m_2}\ar[r]^{y_1}&{C_1}\ar@{=}[d]\\
		{B_2}\ar[d]^{y_2}\ar[r]^{m_1}&M\ar[r]^{e_1}\ar[d]^{e_2}&{C_1}\\
		{C_2}\ar@{=}[r]&{C_2}
	}
	$$
	\noindent which satisfies $\mathfrak{s}(x_{2_*}\delta_1)=[B_2\stackrel{m_1}{\longrightarrow}M\stackrel{e_1}{\longrightarrow}C_1]$ and $\mathfrak{s}(x_{1_*}\delta_2)=[B_1\stackrel{m_2}{\longrightarrow}M\stackrel{e_2}{\longrightarrow}C_2]$.
\end{lem}

Now we are in the position to introduce the concept for the proper classes of $\mathbb{E}$-triangles following \cite{JDP}. In the following part of this section, we always assume that $(\mathcal{C},\mathbb{E},\mathfrak{s})$ is an extriangulated categroy.
\begin{defn}
	Let $\xi$ be a class of $\mathbb{E}$-triangles. One says $\xi$ is \emph{closed under base change} if for any $\mathbb{E}$-triangle
	$$
	\xymatrix{
		A\ar[r]^x&B\ar[r]^y&C\ar@{-->}[r]^{\delta}&
	}
	\in \xi
	$$
	\noindent and any morphism $c:C'\rightarrow C$, then any $\mathbb{E}$-triangle$\xymatrix{A\ar[r]^{x'}&B'\ar[r]^{y'}&C'\ar@{-->}[r]^{c^*\delta}&}$ belongs to $\xi$.
	
	Dually, one says $\xi$ is \emph{closed under cobase change} if for any $\mathbb{E}$-triangle
	$$
	\xymatrix{
		A\ar[r]^x&B\ar[r]^y&C\ar@{-->}[r]^{\delta}&
	}
	\in \xi
	$$
	\noindent and any morphism $a:A\rightarrow A'$, then any $\mathbb{E}$-triangle$\xymatrix{A'\ar[r]^{x'}&B'\ar[r]^{y'}&C\ar@{-->}[r]^{a_*\delta}&}$ belongs to $\xi$.
	
\end{defn}

\begin{defn}
	A class of $\mathbb{E}$-triangles $\xi$ is called \emph{saturated} if in the situation of Lemma \ref{BH}(1), when $\xymatrix{A_2\ar[r]^{x_2}&B_2\ar[r]^{y_2}&C\ar@{-->}[r]^{\delta_2}&}$ and $\xymatrix{A_1\ar[r]^{m_1}&M\ar[r]^{m_1}&B_2\ar@{-->}[r]^{y_2^*\delta_1}&}$ belong to $\xi$, then the
	$\mathbb{E}$-triangle $\xymatrix{A_1\ar[r]^{x_1}&B_1\ar[r]^{y_1}&C\ar@{-->}[r]^{\delta_1}&}$ belongs to  $\xi$.
\end{defn}

We denote the full subcategory consisting of the split $\mathbb{E}$-triangle by $\Delta_0$.
\begin{defn}\label{ZL}\citep[Definition 3.1]{JDP}
	Let $\xi$ be a class of $\mathbb{E}$-triangles which is closed under isomorphisms. $\xi$ is called a \emph{proper class } of $\mathbb{E}$-triangles if the following conditions holds:
	
	(1) $\xi$ is closed under finite coproducts and $\Delta_0 \subseteq \xi$.
	
	(2) $\xi$ is closed under base change and cobase change.
	
	(3) $\xi$ is saturated.
\end{defn}

\begin{defn}\citep[Definition 3.4]{JDP}
	Let $\xi$ be a proper class of $\mathbb{E}$-triangles. A morphism $x$ is called \emph{$\xi$-inflation} if there exists an $\mathbb{E}$-triangle $$\xymatrix{A\ar[r]^x&B\ar[r]^y&C\ar@{-->}[r]^{\delta}&} \in\xi.$$
	Dually, A morphism $y$ is called \emph{$\xi$-deflation} if there exists an $\mathbb{E}$-triangle
	$$\xymatrix{A\ar[r]^x&B\ar[r]^y&C\ar@{-->}[r]^{\delta}&}\in\xi.$$
	
\end{defn}

\section{$\xi$-$\mathcal{G}$projective objects}

\quad~Throughout  this section, we assume that $\xi$ is a  proper class of $\mathbb{E}$-triangles in an extriangulated category $(\mathcal{C},\mathbb{E},\mathfrak{s})$.
\begin{defn}\citep[Definition 4.1]{JDP} An object $P\in \mathcal{C}$ is called $\xi$-\emph{projective} if for any $\mathbb{E}$-triangle
	
	$$
	\xymatrix{
		A\ar[r]^{x}&B\ar[r]^{y}&C\ar@{-->}[r]^{\delta}&
	}
	$$
	in $\xi$, the induced sequence of abelian groups
	$$
	0\longrightarrow\mathcal{C}(P,A)\longrightarrow \mathcal{C}(P,B)\longrightarrow\mathcal{C}(P,C)\longrightarrow 0
	$$
	is exact. We denote by $\mathcal{P}(\xi)$ the subcategory of $\xi$-projective objects in $\mathcal{C}$.
\end{defn}
\begin{rem}\label{REM1}
	(1) $\mathcal{P}(\xi)$ is a full, additive, closed under isomorphism, direct sum and direct summands.
	
	(2) For any $\mathbb{E}$-triangle $\xymatrix{A\ar[r]^{x}&B\ar[r]^{y}&P\ar@{-->}[r]^{\delta}&}$ in $\xi$ with $P\in \mathcal{P}(\xi)$ is split. That means $B\simeq A\oplus P$.
\end{rem}
\begin{proof}
	
	(1) It can be obtained directly from the definition.
	
	(2) Applying functor $\mathcal{C}(P,-)$ to the above $\mathbb{E}$-triangle, we get the exact sequence
	$$
	0\longrightarrow\mathcal{C}(P,A)\stackrel{\mathcal{C}(P,x)}\longrightarrow \mathcal{C}(P,B)\stackrel{\mathcal{C}(P,y)}\longrightarrow\mathcal{C}(P,P)\longrightarrow0
	$$
	since $P\in \mathcal{P}(\xi)$. This implies that $y$ is a retraction. Then $\xymatrix{A\ar[r]^{x}&B\ar[r]^{y}&P\ar@{-->}[r]^{\delta}&}$ is split by Lemma \ref{FJ}.
\end{proof}

An extriangulated category $(\mathcal{C},\mathbb{E},\mathfrak{s})$ is said to have \emph{enough $\xi$-projectives } provided that for each object $A$ there exists an $
\mathbb{E}$-triangle $K\longrightarrow P\longrightarrow A\dashrightarrow$ in $\xi$ with $P\in\mathcal{P}(\xi)$.

The following lemma is used frequently in this thesis.

\begin{lem}\label{FBZH1}\citep[Lemma 4.2]{JDP}
	If $\mathcal{C}$ has enough $\xi$-projectives, then an $\mathbb{E}$-triangle $ A\longrightarrow B\longrightarrow C\dashrightarrow$ in $\xi$ if and only if induced sequence of abelian groups
	$$
	0\longrightarrow\mathcal{C}(P,A)\longrightarrow \mathcal{C}(P,B)\longrightarrow\mathcal{C}(P,C)\longrightarrow0
	$$
	is exact for all $P\in\mathcal{P}(\xi)$.
\end{lem}
The \emph{$\xi$-projective dimension} $\xi$-pd$A$ of an object $A$ is defined inductively. When $A=0$, put $\xi$-pd$A=-1$.  If $A\in\mathcal{P}(\xi)$, then define $\xi$-pd$A=0$. Next by induction, for an integer $n>0$, put $\xi$-pd$A\leqslant n$ if there exists an $\mathbb{E}$-triangle $K\rightarrow P\rightarrow A\dashrightarrow$ in $\xi$ with $P\in \mathcal{P}(\xi)$ and $\xi$-pd$K\leqslant n-1$.

We define $\xi$-pd$A=n$ if $\xi$-pd$A\leqslant n$ and $\xi$-pd$A \nleqslant  n-1$. If $\xi$-pd$A\neq n$, for all $n\geqslant0$, we set  $\xi$-pd$A=\infty$.

\begin{defn}\citep[Definition 4.4]{JDP}
	An complex $\mathbf{X}$ is called \emph{$\xi$-exact} if $\mathbf{X}$ is a diagram
	$$\xymatrix{
		\cdots\ar[r] &X_1\ar[r]^{d_1}&X_0 \ar[r]^{d_0}&X_{-1}\ar[r]&\cdots
	}
	$$
	in $\mathcal{C}$ such that for each integer $n$, there exists an $\mathbb{E}$-triangle $K_{n+1}\stackrel{g_n}\longrightarrow X_n\stackrel{f_n}\longrightarrow C\stackrel{\delta_n}\dashrightarrow$ in $\xi$ and $d_n=g_{n-1}f_n$. These $\mathbb{E}$-triangles are called the\emph{ resolution $\mathbb{E}$-triangles }of the $\xi$-exact complex $\mathbf{X}$.
\end{defn}

\begin{prop}[Schanuel's Lemma]
	For any integer $n\geqslant 0$, if there are two $\xi$-exact complexes in $\Mcc$ as follows
	$$\xymatrix{
		\mathbf{P}:0\ar[r]&K_n\ar[r]^{f_n}&P_{n-1}\ar[r]^{f_{n-1}}&{\cdots}\ar[r]^{f_2}&P_{1}\ar[r]^{f_1}&P_{0}\ar[r]^{f_0}&A\ar[r]&0,\\
		\mathbf{P'}:0\ar[r]&K'_n\ar[r]^{f'_n}&P'_{n-1}\ar[r]^{f'_{n-1}}&{\cdots}\ar[r]^{f'_2}&P'_{1}\ar[r]^{f'_1}&P'_{0}\ar[r]^{f'_0}&A\ar[r]&0}
	$$
	with $P_i$ and $P'_i$ in $\Mcp(\xi)$ for any $0\leqslant i \leqslant n-1$. Then we have
	$$ K_n\oplus P'_{n-1}\oplus P_{n-2}\oplus P'_{n-3}\oplus \cdots \oplus H \simeq K'_n\oplus P_{n-1}\oplus P'_{n-2}\oplus P_{n-3}\oplus \cdots \oplus H'.$$
	Precisely, if $n$ is even, then $H=P_0$, $H'=P'_0$ and if $n$ is odd, then $H=P'_0$, $H'=P_0$.
\end{prop}
\begin{proof}
	If $n=0$, it is obviously true form \citep[Proposition 4.3]{JDP}. Assume that this conclusion is true when $n=k-1$, then we consider the situation with $n=k$.
	
	There are fours $\Mbe$-triangles in $\xi$ since $\mathbf{P}$ and $\mathbf{P'}$ are $\xi$-exact complexes.
	\begin{gather*}
	\xymatrix{K_1\ar[r]^{x_1}&P_0\ar[r]^{y_0}&A\ar@{-->}[r]&}, \xymatrix{K_2\ar[r]^{x_2}&P_1\ar[r]^{y_1}&K_1\ar@{-->}[r]&}\\
	\xymatrix{K'_1\ar[r]^{x'_1}&P'_0\ar[r]^{y'_0}&A\ar@{-->}[r]&}, \xymatrix{K'_2\ar[r]^{x'_2}&P'_1\ar[r]^{y'_1}&K'_1\ar@{-->}[r]&}
	\end{gather*}
	Then we have $K_1\oplus P'_0 \simeq K'_1\oplus P_0$, so we get two sequences as follows
	$$\xymatrix{
		\mathbf{\hat{P}}:0\ar[r]&K_n\ar[r]^{f_n}&P_{n-1}\ar[r]^{f_{n-1}}&{\cdots}\ar[r]^{f_3}&P_{2}\ar[r]^{f_2}&{P_1\oplus P'_0}\ar[r]^{\alpha}&{K_1\oplus P'_0}\ar[r]&0,\\
		\mathbf{\hat{P}'}:0\ar[r]&K'_n\ar[r]^{f'_n}&P'_{n-1}\ar[r]^{f'_{n-1}}&{\cdots}\ar[r]^{f'_3}&P'_{2}\ar[r]^{f'_2}&{P'_1\oplus P_{0}}\ar[r]^{\alpha'}&{K'_1\oplus P_0}\ar[r]&0}
	$$
	where $f=\begin{tiny}\begin{bmatrix}f_2 \\0\end{bmatrix}\end{tiny}, \alpha=\begin{tiny}\begin{bmatrix}y_1 \\1\end{bmatrix}\end{tiny}, f'=\begin{tiny}\begin{bmatrix}f'_2 \\0\end{bmatrix}\end{tiny}, \alpha'=\begin{tiny}\begin{bmatrix}y'_1 \\1\end{bmatrix}\end{tiny}$.
	
	It is easy to see that $\mathbf{\hat{P}}$ and $\mathbf{\hat{P}'}$ are $\xi$-exact complexes in $\Mcc$. Then we have following isomorphism
	$$ K_n\oplus P'_{n-1}\oplus P_{n-2}\oplus P'_{n-3}\oplus \cdots \oplus H \simeq K'_n\oplus P_{n-1}\oplus P'_{n-2}\oplus P_{n-3}\oplus \cdots \oplus H'$$
	by hypothesis, which is desired.
\end{proof}

If $\xymatrix{ K\ar[r]&P\ar[r]&C\ar@{-->}[r]&}$ is an $\mathbb{E}$-triangle in $\xi$ with $P\in\P$, then we call the object $K$ a \emph{first $\xi$-syzygy} of $C$. An \emph{$n$th $\xi$-syzygy} of $C$ is defined as usual by induction. 
\begin{cor}\label{Sch}
	Any two $n$th $\xi$-syzygy of any object $C\in \Mcc$ are $\xi$-projectively equivalent for any $n\geqslant1$.
\end{cor}
\begin{proof}
	By Schanuel's Lemma, it is obvious.
\end{proof}
\begin{defn}\citep[Definition 4.5, 4.6]{JDP}
	Let $\mathcal{W}$ be a class of objects in $\Mcc$. An $\mathbb{E}$-triangle $A\longrightarrow B\longrightarrow C\dashrightarrow$ in $\xi$ is called to be $\Mcc(-,\mathcal{W})$\emph{-exact} (respectively $\Mcc(\mathcal{W},-)$\emph{-exact}) if for any $W\in\mathcal{W}$, the induced sequence of abelian group $0 \rightarrow \Mcc(C,W)\rightarrow\Mcc(B,W)\rightarrow\Mcc(A,W)\rightarrow0$ (respectively $0 \rightarrow \Mcc(W,A)\rightarrow\Mcc(W,B)\rightarrow\Mcc(W,C)\rightarrow0$) is exact in $\mathbf{Ab}$.

	A complex $\mathbf{X}$ is called $\Mcc(-,\mathcal{W})$-\emph{exact}  ( respectively $\Mcc(\mathcal{W},-)$-\emph{exact} ) if it is a $\xi$-exact complex with $\Mcc(-,\mathcal{W})$-exact resolution $\mathbb{E}$-triangles ( respectively $\Mcc(\mathcal{W},-)$-exact resolution $\mathbb{E}$-triangles ).
	
	A $\xi$-exact complex $\mathbf{X}$ is called \emph{complete $\Mcp(\xi)$-exact} if it is $\Mcc(-,\Mcp(\xi))$-exact.
\end{defn}
\begin{defn}
	An \emph{$\xi$-projective resolution} of an object $A\in\Mcc$ is a $\xi$-exact complex
	\[
	\xymatrix{
		\cdots\ar[r]&P_{n}\ar[r]&P_{n-1}\ar[r]&{\cdots}\ar[r]&P_{1}\ar[r]&P_{0}\ar[r]&A\ar[r]&0
	}
	\]
	in $\Mcc$ with $P_{n}\in\P$ for all $n\geqslant0$.
\end{defn}
\begin{defn}\citep[Definition 4.7, 4.8]{JDP}
	A \emph{complete $\xi$-projective resolution} is a complete $\Mcp(\xi)$-exact complex
	$$
	\mathbf{P}: \xymatrix{ \cdots\ar[r]&P_1\ar[r]^{d_1}&P_{0}\ar[r]^{d_0}&P_{-1}\ar[r]&{\cdots}}
	$$
	in $\Mcc$ such that $P_n$ is projective for each integer $n$ . And for any $P_n$, there exists a $\Mcc(-,\Mcp(\xi))$-exact $\mathbb{E}$-triangle $\xymatrix{ K_{n+1}\ar[r]^{g_n}&P_n\ar[r]^{f_n}&K_n\ar@{-->}[r]^{\delta_n}&}$ in $\xi$ which is the resolution $\mathbb{E}$-triangle of $\mathbf{P}$. Then the objects $K_n$ are called \emph{$\xi$-$\Mcg$projective} for each integer $n$. We denote by $\Mcg\Mcp(\xi)$ the subcategory of $\xi$-$\Mcg$projective objects in $\Mcc$.
\end{defn}

Next, we will introduce some fundamental properties of  $\Gproj$ objects. We always assume that the extriangulated category $\extri$ has enough $\xi$-projectives  and satisfies Condition (WIC) for the rest part of this section.

\begin{cond}[Condition (WIC)]\label{WIC}
	Consider the following conditions.
	
	(1) Let $f\in\Mcc(A,B),g\in\Mcc(B,C)$ be any composable  pair of morphisms. If $gf$ is an inflation, then so is $f$.
	
	(2) Let $f\in\Mcc(A,B),g\in\Mcc(B,C)$ be any composable  pair of morphisms. If $gf$ is a deflation, then so is $g$.
\end{cond}

\begin{exam}
	(1) If $\Mcc$ is an exact category, then Condition (WIC) is equivalent to $\Mcc$ is weakly idempotent complete (see \citep[Proposition 7.6]{TB}).
	
	(2) If $\Mcc$ is a triangulated category, then Condition (WIC) is automaticlly satisfied.
\end{exam}

\begin{prop}\label{INDE}\citep[Proposition 4.13]{JDP}
	Let $f\in\Mcc(A,B),g\in\Mcc(B,C)$ be any composable  pair of morphisms. We have that
	
	(1) if $gf$ is a $\xi$-inflation, then so is $f$.
	
	(2) if $gf$ is a $\xi$-deflation, then so is $g$.
\end{prop}

\begin{lem}\label{ABC}\citep[Theorem 4.16]{JDP}
	If $\xymatrix{A\ar[r]^{x}&B\ar[r]^y&C\ar@{-->}[r]^{\delta}&}$ is an $\Mbe$-triangle in $\xi$ with $C\in\GP$, then $A\in \GP$ if and only if $B\in\GP$.
\end{lem}

The \emph{$\Gproj$ dimension} $\Gpd A$ of an object $A$ is defined inductively. When $A=0$, put $\Gpd A=-1$.  If $A\in\GP$, then define $\Gpd A=0$. Next by induction, for an integer $n>0$, put $\Gpd A\leqslant n$ if there exists an $\mathbb{E}$-triangle $K\rightarrow G\rightarrow A\dashrightarrow$ in $\xi$ with $G\in \GP$ and $\Gpd K\leqslant n-1$.

We define $\Gpd A=n$ if $\Gpd A\leqslant n$ and $\Gpd A \nleqslant  n-1$. If $\Gpd A\neq n$, for all $n\geqslant0$, we set  $\Gpd A=\infty$.

Let $\widehat{\mathcal{GP}}(\xi)$ (respectively $\widehat{\Mcp}(\xi)$) denote the full subcategory of $\Mcc$ whose objects are of finite $\Gproj$ (respectively $\xi$-projective) dimension.
\begin{prop}\label{CLOD}\citep[Theorem 4.17]{JDP}
	$\GP$ is closed under direct sums and direct summands.
\end{prop}
\begin{lem}\label{HZG}
	Let $A\in \widehat{\mathcal{GP}}(\xi), G\in\GP$, then $\Gpd (A\oplus G)\leqslant\Gpd A$;
\end{lem}
\begin{proof}
	Let $\Gpd A=n$, then there exists an $\Mbe$-triangle $K\longrightarrow G_A\longrightarrow A\dashrightarrow$ in $\xi$  where $G_A\in\GP$ and $\Gpd K\leqslant n-1$.
	
	Note that the $\Mbe$-triangle $ 0\longrightarrow G\stackrel{1}\longrightarrow G\dashrightarrow$ is in $\xi$ since it is split. So we have the $\Mbe$-triangle
	$$\xymatrix{ K\ar[r]&G_A\oplus G\ar[r]&A\oplus G\ar@{-->}[r]&}$$
	in $\xi$ since $\xi$ is closed under finite direct sums. Because $G_A$ and $G$ are both in $\GP$, then $G_A\oplus G\in\GP$ by Proposition \ref{CLOD}. Hence, $\Gpd (A\oplus G)\leqslant n$ by definition of $\Gproj$ dimension, i.e.
	$$\Gpd (A\oplus G)\leqslant\Gpd A$$
\end{proof}

\begin{cor}{yxh}
	If $\Gpd A\leqslant n$,then there exists an $\Mbe$-triangle $$\xymatrix{ K\ar[r]&P\ar[r]&A\ar@{-->}[r]&}$$ in $\xi$ where $P\in\P$ and $\Gpd K \leqslant n-1$.
\end{cor}
\begin{proof}
	There exists an $\Mbe$-triangle $\xymatrix{ K_A\ar[r]^g&G\ar[r]^f&A\ar@{-->}[r]^{\delta}&}$ in $\xi$ ,where $G$ is in $\GP$ and $\Gpd K_A\leqslant n-1$ since  $\Gpd A\leqslant n$. Because $\Mcc$ has enough $\xi$-projectives, there exists an $\Mbe$-triangle $\xymatrix{ K\ar[r]^{g'}&P\ar[r]^{f'}&A\ar@{-->}[r]^{\delta'}&}$ in $\xi$ with $P\in\P$.
	$$\xymatrix{
		K\ar[r]^{g'}\ar@{-->}[d]^x&P\ar[r]^{f'}\ar@{-->}[d]^y &A\ar@{=}[d]\ar@{-->}[r]^{\delta'}&\\
		K_A\ar[r]^g&G\ar[r]^f&A\ar@{-->}[r]^{\delta}&
	}
	$$
	Since $P\in\P$, there exists a morphism  $y\in \Mcc(P,G)$ such that $gf=f'$. By \citep[Lemma 3.6]{JDP}, there exists a morphism $x\in\Mcc(K,K_A)$ which gives a morphism of $\Mbe$-triangles and an $\Mbe$-triangle
	$$ K\stackrel{\begin{tiny}\begin{bmatrix}
		-x\\
		g'
		\end{bmatrix}\end{tiny}}\longrightarrow K_A \oplus P \stackrel{\begin{tiny}\begin{bmatrix}
		g\ y
		\end{bmatrix}\end{tiny}}\longrightarrow G\stackrel{f^*\delta'}\dashrightarrow
	$$
	which is in $\xi$ since $\xi$ is closed under base change. Then one can get that $$\Gpd K=\Gpd( K_A \oplus P)\leqslant n-1$$ by Lemma \ref{HZG} and \citep[Lemma 5.1]{JDP}.
\end{proof}

\begin{lem}\label{FBZH}
	If $\xymatrix{A\ar[r]^{x}&B\ar[r]^y&C\ar@{-->}[r]^{\delta}&}$ is an $\Mbe$-triangle in $\xi$ with $C\in\GP$, then it is $\Mcc(-,\widehat{\Mcp}(\xi))$-exact. Particularly, it is $\fbzh$.
\end{lem}
\begin{proof}
	See the proof of \citep[Lemma 5.3]{JDP}.
\end{proof}

\begin{defn}\label{Gpr}
	A \emph{$\Gproj$ resolution} of an object $A\in\Mcc$ is a $\xi$-exact complex
	\[
	\xymatrix{
		\cdots\ar[r]&G_{n}\ar[r]&G_{n-1}\ar[r]&{\cdots}\ar[r]&G_{1}\ar[r]&G_{0}\ar[r]&A\ar[r]&0
	}
	\]
	in $\Mcc$ such that $G_{n}\in\GP$ for all $n\geqslant0$.
\end{defn}

\begin{thm}\label{ZHFJ}
	Let any $A$ be a object in $\Mcc$. Then $A$ has a $\xi$-projective resolution which is $\mathcal{C}(-,\P)$-exact if and only if $A$ has a $\Gproj$ resolution which is $\fbzh$.
\end{thm}
\begin{proof}
	The ``if" part is obvious since $\P\subseteq\GP$. Assume that $A$ has a $\Gproj$ resolution which is $\fbzh$. Then there exists an $\Mbe$-triangle $K_1\stackrel{g_0}{\longrightarrow}G_0\stackrel{f_0}{\longrightarrow}A\stackrel{\delta_0}{\dashrightarrow}$ which is $\fbzh$, where $G_0\in\GP$ and $K_1$ has a $\Gproj$ resolution which is $\fbzh$. So there exists an $\mathbb{E}$-triangle $\xymatrix{G'_o\ar[r]&P_0\ar[r]&G_0\ar@{-->}[r]&}$ such that $G'_0\in \GP,P_0\in\P$, which is $\mathcal{C}(-,\mathcal{P}(\xi))$-exact. By $\rm (ET4)^{op}$, there exists a commutative diagram:
	$$
	\xymatrix{
		G'_0 \ar@{=}[d]\ar[r]&E\ar[d]\ar[r]& K_1\ar[d]\ar@{-->}[r]& \\
		G'_0 \ar[r]&P_0\ar[d]\ar[r]& G_0\ar[d]\ar@{-->}[r]& \\
		& A\ar@{-->}[d]\ar@{=}[r] & A\ar@{-->}[d]\\
		&&   }
	$$
	
	Note that $\xymatrix{G'_0\ar[r]&E\ar[r]& K_1\ar@{-->}[r]& }$ is an $\Mbe$-triangle in $\xi$ since $\xi$ is closed under base change. Applying the functor $\mathcal{C}(\mathcal{P}(\xi),-)$ to the above diagram, it is easy to see that the $\Mbe$-triangle $\xymatrix{E\ar[r]&P_0\ar[r]&A\ar@{-->}[r]&}$ is $\mathcal{C}(\mathcal{P}(\xi),-)$-exact by a diagram chasing. Hence it is in $\xi$ by Lemma \ref{FBZH1}.  Applying the functor $\mathcal{C}(-,\mathcal{P}(\xi))$ to the above diagram, it is also easy to see that
	$$\xymatrix{E\ar[r]&P_0\ar[r]&A\ar@{-->}[r]&} \  \text{and}\  \xymatrix{G'_0\ar[r]&E\ar[r]&K_1\ar@{-->}[r]&}$$
	are $\fbzh$ by a diagram chasing. Since $K_1$ has a $\Gproj$ resolution which is $\fbzh$, there exists an $\Mbe$-triangle $K_2\longrightarrow G_1\longrightarrow K_1 \dashrightarrow$ which is $\Mcc(-,\Mcp(\xi))$-exact, where $G_1\in\GP$, and $K_2$ has a $\Gproj$ resolution which is $\fbzh$. By Lemma \ref{BH}, there exists following commutative diagram:
	$$
	\xymatrix{
		&G'_0\ar@{=}[r]\ar[d]&G'_0\ar[d]\\
		K_2\ar[r]\ar@{=}[d]&M\ar[r]\ar[d]&E\ar[d]\ar@{-->}[r]&\\
		K_2\ar[r]&G_1\ar@{-->}[d]\ar[r]&K_1\ar@{-->}[r]\ar@{-->}[d]&\\
		&&
	}
	$$
	The $\Mbe$-triangles $\xymatrix{K_2\ar[r]&M\ar[r]&E\ar@{-->}[r]&}$ and $\xymatrix{G'_0\ar[r]&M\ar[r]&G_1\ar@{-->}[r]&}$ are in $\xi$ since $\xi$ is closed under base change. It implies $M\in\GP$ by Lemma \ref{ABC} because of $G'_0\in\GP$ and $G_1\in\GP$. Applying the functor $\fbzh$ to the above diagram, it is not hard to get that the $\Mbe$-triangle $\xymatrix{K_2\ar[r]&M\ar[r]&E\ar@{-->}[r]&}$ is $\fbzh$ by a diagram chasing. Proceeding in this manner, we can obtain a $\fbzh$ $\xi$-projective resolution of $A$.
\end{proof}

Let $\mathcal{G}^0\mathcal{P}(\xi)=\P$ and $\mathcal{G}^1\mathcal{P}(\xi)=\GP$. For any $n\geqslant 1$, let  $\mathcal{G}^{n+1}\mathcal{P}(\xi)=\mathcal{G}^{n}\mathcal{P}(\xi)$. Then we have a corollary as follows.
\begin{cor}
	For any $n\geqslant 1$, one can get that $\mathcal{G}^{n}\mathcal{P}(\xi)=\GP$.
\end{cor}

\begin{proof}
	It is obvious that $\P\subseteq \GP\subseteq\cdots\subseteq{\mathcal{G}}^{n}
	\mathcal{P}(\xi)\subseteq {\mathcal{G}}^{n+1}\mathcal{P}(\xi)\subseteq{\cdots}$ by definition.
	
	For any $A\in\mathcal{G}^2\Mcp(\xi)$, there exist a $\fbzh$ $\Mbe$-triangle
	$$\xymatrix{K_{n+1}\ar[r]&G_n\ar[r]&K_n\ar@{-->}[r]&}$$
	for any $n\geqslant0$ such that $G_n\in\GP$ and $K_0=A$. Since $\Mcp(\xi)\subseteq\GP$, then we have the following complex $\mathbf{G}$
	$$
	\xymatrix@C=0.5cm@R=0.5cm{
		\mathbf{G:}\quad\cdots \ar[r]&G_2 \ar[rr] \ar[dr]&  &   G_1 \ar[rr]\ar[dr]& &G_0 \ar[dr]\ar[rr]&&A\ar[r]&0  \\
		K_3\ar[ur] &  &K_2\ar[ur]   &&K_1\ar[ur] &&A\ar@{=}[ur]            }
	$$
	which is a $\fbzh$ $\Gproj$ resolution of $A$. By Theorem \ref{ZHFJ}, $A$ has a $\xi$-projective resolution which is $\fbzh$. It is implies that $A\in\GP$. Hence, we have $\mathcal{G}^2\mathcal{P}(\xi)=\GP$. By using induction on $n$, we get
	$$ \mathcal{G}^{n}\mathcal{P}(\xi)=\GP$$
	for any integer $n\geqslant1$.
\end{proof}

At the end of this section, we give some inequalities of $\Gproj$ dimension in an $\Mbe$-triangle. Firstly, we have following lemma.
\begin{lem}[Horseshoe Lemma]\label{HS}
	Let $\xymatrix{A\ar[r]&B\ar[r]&C\ar@{-->}[r]&}$ be an $\Mbe$-triangle in $\xi$. Then there are $\xi$-projective resolutions $\mathbf{P}_A,\mathbf{P}_B$ and $\mathbf{P}_C$ of $A,B$ and $C$, respectively, and a commutative diagram
	$$\xymatrix{
		\mathbf{P}_A\ar[d]\ar[r]^{x^{\bullet}}&\mathbf{P}_B\ar[d]\ar[r]^{y^{\bullet}}&\mathbf{P}_C\ar[d]\ar@{-->}[r]&\\
		A\ar[r]&B\ar[r]&C\ar@{-->}[r]&
	}
	$$
	such that $\xymatrix{P^n_A\ar[r]^{x^n}&P^n_B\ar[r]^{y^n}&P^n_C\ar@{-->}[r]&}$ is a split $\Mbe$-triangle, i.e. $P^n_B\simeq P^n_A\oplus P^n_C$ for any $n\geqslant0$.
\end{lem}
\begin{proof}
	It is easy to see that this lemma holds by \citep[Lemma 4.14]{JDP} and we can also see this lemma in \citep[Lemma 3.3]{JZP}.
\end{proof}

\begin{thm}\label{XMYY}
	Let $\xymatrix{ A\ar[r]&B\ar[r]&C\ar@{-->}[r]&}$ be an $\Mbe$-triangle in $\xi$, then there exist following inequalities.
	
	(1) $\Gpd B\leqslant \max\{\Gpd A, \Gpd C\}$;
	
	(2) $\Gpd A\leqslant \max\{\Gpd B, \Gpd C -1\}$;
	
	(3) $\Gpd C\leqslant \max\{\Gpd B, \Gpd A +1\}$.
\end{thm}

\begin{proof}
	We always assume that the right side of above inequalities are finite, because that is trivial when they are  infinite.
	
	(1) Let $\Gpd A\leqslant n, \  \Gpd C\leqslant m, \  t=\max\{m,n\}$. And let
	\begin{gather*}
	\xymatrix{\cdots\ar[r]&P_A^i\ar[r]&P_A^{i-1}\ar[r]&{\cdots}\ar[r]&P^0_A\ar[r]&0},\\
	\xymatrix{\cdots\ar[r]&P_C^i\ar[r]&P_C^{i-1}\ar[r]&{\cdots}\ar[r]&P^0_C\ar[r]&0}
	\end{gather*}
	are $\xi$-projective resolutions of $A$ and $C$, respectively. Then we have following commutative diagram by Horseshoe Lemma and \citep[Lemma 4.14]{JDP}.
	$$
	\xymatrix{
		{K_A^t}\ar[r]^{f}\ar[d]^{g_A^t}&{K_B^t}\ar[r]\ar[d]^{g_B^t}&{K_C^t}\ar@{-->}[r]\ar[d]&\\
		{P_A^{t-1}}\ar[r]^l\ar[d]&{P_A^{t-1}\oplus P_C^{n-1}}\ar[d]\ar[r]&{P_C^{t-1}}\ar@{-->}[r]\ar[d]&\\
		{\vdots}\ar[d]&{\vdots}\ar[d]&{\vdots}\ar[d]&\\
		{P_A^{0}}\ar[r]\ar[d]&{P_A^{0}\oplus P_C^{0}}\ar[d]\ar[r]&{P_C^{0}}\ar@{-->}[r]\ar[d]&\\
		A\ar[r]&B\ar[r]&C\ar@{-->}[r]&\\
	}
	$$
	Where $K_A^t,K_B^t$ and $ K_C^t$ are $t$th $\xi$-syzygy of $A,B$ and $C$, respectively.
	Then $f$ is $\xi$-inflation by Proposition \ref{INDE}, since $g_B^tf=lg_A^t$ with $l$ and $g_A^t$ being $\xi$-inflation.  It is easy to check that the $\Mbe$-triangle $\xymatrix{K_A^t\ar[r]&K_B^t\ar[r]&K_C^t\ar@{-->}[r]&}$ is isomorphism to an $\Mbe$-triangle in $\xi$ by \citep[Corollary 3.6(3)]{HY}, hence it is an $\Mbe$-triangle in $\xi$.
	Note that $\Gpd A\leqslant t$ and $\Gpd C\leqslant t$. Then $K_A^t$ and $K_C^t$ are $\Gproj$ by \citep[Proposition 5.2]{JDP}. So one can get that $K_C^t$ is $\Gproj$ by Lemma \ref{ABC} and therefore there exists that
	$$\Gpd B\leqslant t=\max\{\Gpd A, \Gpd C\}$$
	by definition of $\Gproj$ dimension.
	
	(2) Let $\Gpd B\leqslant n, \Gpd C\leqslant m$ and $t=\max\{m-1,n\}$. Then there exists an $\Mbe$-triangle $\xymatrix{K\ar[r]&G\ar[r]&C\ar@{-->}[r]&}$ in $\xi$ where $G\in \GP$ and $\Gpd K\leqslant m-1$. By Lemma \ref{BH} there is a following commutative diagram:
	$$\xymatrix{
		&K\ar@{=}[r]\ar[d]&K\ar[d]\\
		A\ar[r]\ar@{=}[d]&M\ar[d]\ar[r]&G\ar[d]\ar@{-->}[r]&\\
		A\ar[r]&B\ar@{-->}[d]\ar[r]&C\ar@{-->}[d]\ar@{-->}[r]&\\
		&&
	}
	$$
	Then $\xymatrix{A\ar[r]&M\ar[r]&G\ar@{-->}[r]&}$ and  $\xymatrix{K\ar[r]&M\ar[r]&B\ar@{-->}[r]&}$ are both $\Mbe$-triangles in $\xi$ since $\xi$ is closed under base change. By (1) we have $\Gpd M\leqslant t$. Because of $G\in GP$ ,then $\Gpd A=\Gpd M\leqslant t$ by \citep[Lemma 5.1]{JDP}. That is to say
	$$\Gpd A\leqslant \max\{\Gpd B, \Gpd C -1\}.$$
	
	(3) Let $\Gpd A\leqslant m,\Gpd B\leqslant n$, and $t=\max\{m+1,n\}$.
	Then there exists an $\Mbe$-triangle $\xymatrix{K\ar[r]&G\ar[r]&B\ar@{-->}[r]&}$ in $\xi$ where $G\in \GP$ and $\Gpd K\leqslant n-1$. By $\rm{(ET4)^{op}}$, there exists following diagram:
	$$\xymatrix{
		K\ar@{=}[d]\ar[r]&D\ar[d]\ar[r]&A\ar[d]\ar@{-->}[r]&\\
		K\ar[r]&G\ar[d]\ar[r]&B\ar[d]\ar@{-->}[r]&\\
		&C\ar@{-->}[d]\ar@{=}[r]&C\ar@{-->}[d]&\\
		&&
	}
	$$
	Then the $\Mbe$-triangle $\xymatrix{K\ar[r]&D\ar[r]&A\ar@{-->}[r]&}$ is in $\xi$ since $\xi$ is closed under base change. It is easy to see that $\xymatrix{D\ar[r]&G\ar[r]&C\ar@{-->}[r]&}$ is $\Mcc(\P,-)$-exact by diagram chasing, so the $\Mbe$-triangle $\xymatrix{D\ar[r]&G\ar[r]&C\ar@{-->}[r]&}$ is in $\xi$.
	
	Because $\Gpd K \leqslant n-1, \Gpd A\leqslant m$, one can get  $\Gpd D\leqslant t-1 $ by (1). So $\Gpd C\leqslant t$ i.e.
	$$\Gpd C\leqslant \max\{\Gpd B, \Gpd A +1\}.$$
	So the proof was completed.
\end{proof}

\begin{cor}
	(1) Let $\xymatrix{ A\ar[r]&B\ar[r]&C\ar@{-->}[r]&}$ be an $\Mbe$-triangle in $\xi$. If the $\xi$-$\mathcal{G}$projective dimension for the two of $A,B$ and $C$ are finite, then so is the left one.
	
	(2) Let $A,B\in\Mcc$, then $\Gpd (A\oplus B)\leqslant \max\{\Gpd A, \Gpd B\}$.
\end{cor}
\begin{proof}
	It is obvious from the Theorem \ref{XMYY}.
\end{proof}

The \emph{$\xi$-injective} objects and \emph{$\xi\text{-}\mathcal{G}\text{injective}$} objects 
are defined by the dual  of $\xi$-projective objects and $\Gproj$ objects respectively. 
 All the results which are mentioned in the previous section concerning $\xi$-projective objects 
 and $\Gproj$ objects have $\xi$-injective  objects and $\xi\text{-}\mathcal{G}\text{injective}$
  objects counterparts; hence, the statements and their proofs of the dual results on
    $\xi$-injective  objects and $\xi\text{-}\mathcal{G}\text{injective}$ objects
are omitted in this thesis.

We denote by $\mathcal{I}(\xi)$ (resp.$\Mcg\mathcal{I}(\xi)$ ) the full subcategory of $\xi$-injective (resp.$\xi$-$\Mcg$injective) objects  in $\Mcc$, and use $\widehat{\mathcal{I}}(\xi)$ (resp. $\widehat{\mathcal{GI}}(\xi)$) to denote the full subcategory of $\Mcc$ whose objects have finite $\xi$-injective (resp. $\xi$-$\mathcal{G}$injective) dimension.

\section{$\xi$-$n$-strongly $\mathcal{G}$projective objects}
\quad~In this section, we introduce some special $\Gproj$ objects in extriangulated category  
which are called  $\xi$-$n$-strongly $\mathcal{G}$projective objects for any integer
 $n\geqslant1$. We study the relation between
  $\xi$-$m$-strongly $\mathcal{G}$projective objects  and $\xi$-$n$-strongly 
  $\mathcal{G}$projective objects whenever $m\neq n$, and give some 
  equivalent characterizations of $\xi$-$n$-strongly $\mathcal{G}$projective objects.

Throughout this section, we assume that
 $\Mcc=(\Mcc,\Mbe,\mathfrak{s})$   has enough $\xi$-projectives and $\xi$-injectives
 satisfying Condition(WIC). We also assume that $m$ and $n$ are 
  positive integers and $n\leqslant m$.
\begin{defn}\citep[Definition 3.2]{JZP}
	Let $A$ and $B$ be objects in $\mathcal{C}$.
	
	(1) If we choose a $\xi$-projective resolution $\xymatrix{\mathbf{P}\ar[r]&A}$ of $A$, then for any  integer $n\geqslant 0$, the $\xi$-cohomology groups $\xi \text{xt}_{\mathcal{P}(\xi)}^n(A, B)$ are defined as
	$$
	\xi \text{xt}_{\mathcal{P}(\xi)}^n(A, B)=H^n(\mathcal{C}(\mathbf{P},B)).
	$$
	
	(2) If we choose a $\xi$-injective coresolution $\xymatrix{B\ar[r]&\mathbf{I}}$ of $B$, then for any  integer $n\geqslant 0$, the $\xi$-cohomology groups $\xi \text{xt}_{\mathcal{I}(\xi)}^n(A, B)$ are defined as
	$$
	\xi \text{xt}_{\mathcal{I}(\xi)}^n(A, B)=H^n(\mathcal{C}(A,\mathbf{I})).
	$$
	
	Then there exists an isomorphism $\xi \text{xt}_{\mathcal{P}(\xi)}^n(A, B)\backsimeq \xi \text{xt}_{\mathcal{I}(\xi)}^n(A, B)$, which is denoted by $\xi \text{xt}_{\xi}^n(A,B)$.
	\end{defn}
	\begin{lem}\label{LZHL}\citep[Lemma 3.4]{JZP}
		If $\xymatrix{
			A\ar[r]^x &B\ar[r]^y &C\ar@{-->}[r]^{\delta}&}$ is an $\mathbb{E}$-triangle in $\xi$, then for any objects $X$  in $\mathcal{C}$, we have the following long exact sequences in $\mathbf{Ab}$
		$$
		\xymatrix{
		0\ar[r]&\ext^0(X,A)\ar[r]&\ext^0(X,B)\ar[r]&\ext^0(X,C)\ar[r]&\ext^1(X,A)\ar[r]&\cdots
		}
		$$
		and
		$$
		\xymatrix{
		0\ar[r]&\ext^0(C,X)\ar[r]&\ext^0(B,X)\ar[r]&\ext^0(A,X)\ar[r]&\ext^1(C,X)\ar[r]&\cdots
		}
		$$
		For any objects $A$ and $B$, there is always a natural map $\delta: \mathcal{C}\rightarrow\ext^0(A,B)$, which is an isomorphism if $A\in \P$ or $B\in \mathcal{I}(\xi)$.
		\end{lem}
	
\begin{defn}\label{strongly}
	Let $n\geqslant1$ be a integer. An object $A\in\Mcc$ is called \emph{$\xi$-$n$-strongly $\mathcal{G}$projective object} ($\xi$-$n$-SG-projective for short) if there exists a $\xi$-exact complex
	$$
	0\longrightarrow A\stackrel{f_n}\longrightarrow P_{n-1}\stackrel{f_{n-1}}\longrightarrow P_{n-2}\longrightarrow{\cdots}\longrightarrow P_1\stackrel{f_1}\longrightarrow P_0\stackrel{f_0}\longrightarrow A\longrightarrow 0
	$$
	with $P_i\in\mathcal{P}(\xi)$ for any $0\leqslant i\leqslant n-1$, which is $\fbzh$. In particular, if $n=1$, we say $A$ is $\xi$-SG-projective.
	
	For any $n\geqslant1$, we denote the full subcategory of all the $\xi$-$n$-SG-projectives by $n$-$\mathcal{SGP}(\xi)$, and denote the full subcategory of all the $\xi$-SG-projectives by $\mathcal{SGP}(\xi)$.
\end{defn}

\begin{rem}\label{REM2}
	(1) For any $n \geqslant 1$, we have $$\mathcal{P}(\xi)\subseteq \mathcal{SGP}(\xi)\subseteq n\text{-}\mathcal{SGP}(\xi)\subseteq\GP.$$
	
	(2) For any $A\in n\text{-}\SGP$, there exists a complete $\mathcal{P}(\xi)$-exact complex
	$$
	\mathbf{A}:\   0\longrightarrow A\stackrel{f_n}\longrightarrow P_{n-1}\stackrel{f_{n-1}}\longrightarrow P_{n-2}\longrightarrow{\cdots}\longrightarrow P_1\stackrel{f_1}\longrightarrow P_0\stackrel{f_0}\longrightarrow A\longrightarrow 0
	$$
	and for each $0\leqslant i\leqslant n-1$, there exists a $\fbzh$ resolution $\Mbe$-triangle of $$\mathbf{A}~: \xymatrix{ K_{i+1}\ar[r]&P_i\ar[r]&K_i\ar@{-->}[r]&}$$ where $K_n=K_0=A$. Then for any $0\leqslant i\leqslant n$, $K_n$ is also $\xi$-$n$-SG-projective.
\end{rem}
\begin{proof}
	It is an immediate consequence from definition.
\end{proof}
\begin{prop}\label{PROP2}
	For any $n\geqslant1$, $\nSGP$ is closed under finite direct sums.
\end{prop}
\begin{proof}
	Let $\{A_j\}_{j\leqslant m}$ be a set of $\xi$-$n$-SG-projectives in $\Mcc$ with integer $j\geqslant1$. Then for any $j\leqslant m$, there exists a complete  $\P$-exact complex:
	$$0\longrightarrow A_j\longrightarrow P_{n-1}^{(j)}\longrightarrow\cdots \longrightarrow P_0^{(j)}\longrightarrow A_j\longrightarrow0$$
	with $P_j^{(j)}\in \P$ for any $0\leqslant i \leqslant n-1$. So we get an $\xi$-exact complex:
	$$0\longrightarrow\oplus_{j\leqslant m}A_j\longrightarrow \oplus_{j\leqslant m}P_{n-1}^{(j)}\longrightarrow\cdots \longrightarrow\oplus_{j\leqslant m}P_0^{(j)}\longrightarrow\oplus_{j\leqslant m}A_j\longrightarrow0.$$
	Because $ \oplus_{j\leqslant m}P_{n-1}^{(j)}\text{,}\cdots\text{,} \oplus_{j\leqslant m}P_0^{(j)} $ are $\xi$-projectives and the obtained $\xi$-exact complex is still $\fbzh$. Then we completed this proof.
\end{proof}
\begin{lem}\label{Lem1}
	If $n\mid m$, then $n$-$\SGP\subseteq m$-$\SGP$.
\end{lem}

\begin{proof}
	Assume that $A\in \nSGP$. Then there exists a complete $\P$-exact complex
	$$
	0\longrightarrow A\stackrel{f_n}\longrightarrow P_{n-1}\stackrel{f_{n-1}}\longrightarrow P_{n-2}\longrightarrow{\cdots}\longrightarrow P_1\stackrel{f_1}\longrightarrow P_0\stackrel{f_0}\longrightarrow A\longrightarrow 0
	$$
	If $n\mid m$, then we can get  a complete $\P$-exact complex
	$$
	0\longrightarrow A\stackrel{f_n}\longrightarrow \mathbf{P}_{m-1}\stackrel{f_nf_0}\longrightarrow \mathbf{P}_{m-2}\longrightarrow{\cdots}\longrightarrow \mathbf{P}_1\stackrel{f_nf_0}\longrightarrow \mathbf{P}_0\stackrel{f_0}\longrightarrow A\longrightarrow 0
	$$
	where $$\mathbf{P}_i=P_{n-1}\stackrel{f_{n-1}}\longrightarrow P_{n-2}\stackrel{f_{n-2}}\longrightarrow{\cdots}\stackrel{f_2} \longrightarrow P_1\stackrel{f_1}\longrightarrow P_0,\  i=0,1,\cdots m-1.$$
	So $A\in \mSGP$. Therefore,  $n$-$\SGP\subseteq m$-$\SGP$.
\end{proof}

\begin{prop}\label{PROP1}
	(1) If $n\mid m$, then $n$-$\SGP\cap m$-$\SGP$= $n$-$\SGP$.
	
	(2) If $n\nmid m$ and $m=kn+l$, where $k$ is a  positive integer and $0<l<n$. Then $$\nSGP\cap \mSGP\subseteq l\text{-}\SGP.$$
	
\end{prop}
\begin{proof}
	(1) It is trivial by Lemma \ref{Lem1}.
	
	(2) By Lemma \ref{Lem1}, we have that $$\nSGP\cap\mSGP\subseteq \mSGP\cap kn\text{-}\SGP.$$
	Assume that $A\in \mSGP\cap kn\text{-}\SGP$. Then there exists a complete $\P$-exact complex
	$$
	0\longrightarrow A\longrightarrow P_{m-1}\longrightarrow P_{m-2}\longrightarrow{\cdots}\longrightarrow P_1\longrightarrow P_0\longrightarrow A\longrightarrow 0
	$$
	with $P_i\in\mathcal{P}(\xi)$ for any $0\leqslant i\leqslant m-1$. For each $0\leqslant i\leqslant m-1$, we have a $\fbzh$ $\Mbe$-triangle $\xymatrix{K_{i+1}\ar[r]&P_i\ar[r]&K_i\ar@{-->}[r]&}$ in $\xi$ where $K_m=K_0=A$, which is the resolution $\Mbe$-triangle of the complex. Because $A\in kn$-$\SGP$, $A$ and $K_{kn}$ are $\xi$-projectively equivalent, that is, there exists $\xi$-projectives $P$ and $Q$ in $\Mcc$, such that $A\oplus P\simeq Q\oplus K_{kn}$ by Schanuel's Lemma.
	
	First, consider the following commutative diagram by Lemma \ref{BH}.
	$$
	\xymatrix{
		&Q\ar@{=}[r]\ar[d]&Q\ar[d]&\\
		K_{kn+1}\ar[r]\ar@{=}[d]&B\ar[d]\ar[r]&A\oplus P\ar[d]\ar@{-->}[r]&\\
		K_{kn+1}\ar[r]&P_{kn}\ar@{-->}[d]\ar[r]&K_{kn}\ar@{-->}[r]\ar@{-->}[d]&\\
		&&
	}
	$$
	Then $Q\longrightarrow B\longrightarrow P_{kn}\dashrightarrow$ and $K_{kn+1}\longrightarrow B\longrightarrow A\oplus P\dashrightarrow$ are $\Mbe$-triangles in $\xi$ since $\xi$ is closed under base change. Note that $ Q\longrightarrow B\longrightarrow P_{kn}\dashrightarrow$ is split, then $B\simeq Q\oplus P_{kn}\in\P$. Applying the functor $\Mcc(-,\Mcp(\xi))$ to the above diagram, we can get that$K_{kn+1}\longrightarrow B\longrightarrow A\oplus P\dashrightarrow$ is $\fbzh$ by a simple diagram chasing.
	
	Next, consider the following commutative diagram by $\rm{(ET4)^{op}}$
	$$
	\xymatrix{
		K_{kn+1}\ar[r]\ar@{=}[d]&C\ar[d]\ar[r]&A\ar[d]\ar@{-->}[r]&\\
		K_{kn+1}\ar[r]&B\ar[d]\ar[r]&A\oplus P\ar@{-->}[r]\ar[d]&\\
		&P\ar@{=}[r]\ar@{-->}[d]&P\ar@{-->}[d]\\
		&&
	}
	$$
	where $K_{kn+1}\longrightarrow C\longrightarrow A\dashrightarrow$ is an $\Mbe$-triangle in $\xi$ since  $\xi$ is closed under base change, and  $C\longrightarrow B\longrightarrow P\dashrightarrow$ is in $\xi$ since it is split by Remark \ref{REM1}(2). Now, a simple diagram chasing shows that
	the  $\Mbe$-triangle $K_{kn+1}\longrightarrow C\longrightarrow A\dashrightarrow$ is $\fbzh$ with $C\in \P$.
	
	Thus we obtain a $\xi$-exact complex as follows
	$$0\longrightarrow A\longrightarrow P_{m-1}\longrightarrow\cdots\longrightarrow P_{kn+1}\longrightarrow C\longrightarrow A\longrightarrow0$$
	which is still $\fbzh$. That is to say  $A$ is in $l$-$\SGP$, hence
	$$\nSGP\cap \mSGP\subseteq l\text{-}\SGP.$$
\end{proof}

We use gcd$(m,n)$ to denote the greatest common divisor of $m$ and $n$, then we have:
\begin{thm}\label{STR1}
	$\mSGP\cap\nSGP=\text{gcd}(m,n)\text{-}\SGP$.
\end{thm}
\begin{proof}
	If $n\mid m$, then this assertion follows from Proposition \ref{PROP1}(1).
	
	If $n\nmid m$, then we can assume that $m=k_0n+l_0$, where $k_0$ is a positive integer and $0<l_0<n$. By Proposition \ref{PROP1}(2), we can get that $$\mSGP\cap\nSGP\subseteq l_0\text{-}\SGP.$$
	If $l_0\nmid n$ and $n=k_1l_0+l_1$ with $0<l_1<l_0$, then by Proposition \ref{PROP1}(2) again, we have that
	$$\mSGP\cap\nSGP\subseteq \nSGP \cap l_0\text{-}\SGP\subseteq l_1\text{-}\SGP.$$
	continuing the above procedure, after finite steps, there exists a positive integer $t$ such that $l_t=k_{t+2}l_{t+1}$ and $ l_{t+1}=\text{gcd}(m,n)$. Then we have
	\begin{equation*}
	\begin{split}
	\mSGP\cap\nSGP& \subseteq l_t\text{-}\SGP\cap l_{t+1}\text{-}\SGP\\
	&=l_{t+1}\text{-}\SGP\\
	&=\text{gcd}(m,n)\text{-}\SGP.
	\end{split}
	\end{equation*}
	
	On the other hand, we have $\text{gcd}(m,n)\text{-}\SGP\subseteq\mSGP\cap\nSGP$ by Lemma \ref{Lem1}. Then we have done this proof.
\end{proof}

\begin{cor}
	For any  integer $n\geqslant1$, $\nSGP\cap (n+1)\text{-}\SGP=\SGP$. In particular, $\bigcap_{n\geqslant 2}n\text{-}\SGP=\SGP$.
\end{cor}

Next, we give some equivalent characterization of $\xi$-$n$-SG-projective.

\begin{thm}\label{STR2}
	Let integer $n\geqslant 1$ and $A\in\Mcc$. Then the following statements are equivalent.
	
	(1) $A$ is  $\xi$-$n$-SG-projective.
	
	(2) There exists a $\xi$-exact complex:
	$$0\longrightarrow A\longrightarrow P_{n-1}\longrightarrow P_{n-2}\longrightarrow\cdots\longrightarrow P_1\longrightarrow P_0\longrightarrow A\longrightarrow0$$
	with $P_i\in\P$ and the resolution $\Mbe$-triangles $\xymatrix{ K_{i+1}\ar[r]&P_i\ar[r]&K_i\ar@{-->}[r]&}$  in $\xi$ for any $ 0\leqslant i \leqslant n-1$ where $K_{n}=K_0=A$, such that $\oplus_{i=1}^nK_i$ is in $\SGP$.
	
	(3)  There exists a $\xi$-exact complex:
	$$0\longrightarrow A\longrightarrow P_{n-1}\longrightarrow P_{n-2}\longrightarrow\cdots\longrightarrow P_1\longrightarrow P_0\longrightarrow A\longrightarrow0$$
	with $P_i\in\P$  and the resolution $\Mbe$-triangles $\xymatrix{ K_{i+1}\ar[r]&P_i\ar[r]&K_i\ar@{-->}[r]&}$  in $\xi$ for any $ 0\leqslant i \leqslant n-1$ where $K_{n}=K_0=A$, such that $\oplus_{i=1}^nK_i$ is in $\GP$.
	
	(4) There exists a $\xi$-exact complex:
	$$0\longrightarrow A\longrightarrow P_{n-1}\longrightarrow P_{n-2}\longrightarrow\cdots\longrightarrow P_1\longrightarrow P_0\longrightarrow A\longrightarrow0$$
	with $\xi\text{-pd} P_i< \infty$ and the resolution $\Mbe$-triangles $\xymatrix{ K_{i+1}\ar[r]&P_i\ar[r]&K_i\ar@{-->}[r]&}$  in $\xi$ for any $ 0\leqslant i \leqslant n-1$ where $K_{n}=K_0=A$, such that $\oplus_{i=1}^nK_i$ is in $\SGP$.
	
	(5)  There exists a $\xi$-exact complex:
	$$0\longrightarrow A\longrightarrow P_{n-1}\longrightarrow P_{n-2}\longrightarrow\cdots\longrightarrow P_1\longrightarrow P_0\longrightarrow A\longrightarrow0$$
	with  $\xi\text{-pd} P_i< \infty$ and the resolution $\Mbe$-triangles $\xymatrix{ K_{i+1}\ar[r]&P_i\ar[r]&K_i\ar@{-->}[r]&}$  in $\xi$ for any $ 0\leqslant i \leqslant n-1$ where $K_{n}=K_0=A$, such that $\oplus_{i=1}^nK_i$ is in $\GP$.
\end{thm}

\begin{proof}
	(1) $ \Rightarrow$ (2) Assume $A$ is $\xi$-$n$-SG-projective, then there exists a complete $\P$-exact complex:
	$$0\longrightarrow A\longrightarrow P_{n-1}\longrightarrow P_{n-2}\longrightarrow\cdots\longrightarrow P_1\longrightarrow P_0\longrightarrow A\longrightarrow0$$
	with $P_i\in\P$ for any $ 0\leqslant i \leqslant n-1$. Thus for each $0\leqslant i\leqslant n-1$, we have a $\fbzh$ resolution $\Mbe$-triangle $\xymatrix{ K_{i+1}\ar[r]&P_i\ar[r]&K_i\ar@{-->}[r]&}$ in $\xi$, where $K_{n}=K_0=A$. 
	By adding those $\Mbe$-triangles, we can get a $\fbzh$ $\Mbe$-triangle in $\xi$ as follows:
	$$\xymatrix{\oplus_{i=1}^n K_i\ar[r]&\oplus_{i=0}^{n-1} P_{i-1}\ar[r]&\oplus_{i=0}^{n-1}K_i\ar@{-->}[r]&}.$$
	It is easy to see $\oplus_{i=1}^n K_i\simeq \oplus_{i=0}^{n-1}K_i$, then it is enough to show that $\oplus_{i=1}^n K_i$ is in $\SGP$.
	
	(2) $\Rightarrow$ (3) $\Rightarrow$ (5) and (2) $\Rightarrow$ (4) $\Rightarrow$ (5) are trivial.
	
	(5) $\Rightarrow$ (1) Let
	$$0\longrightarrow A\longrightarrow P_{n-1}\longrightarrow P_{n-2}\longrightarrow\cdots\longrightarrow P_1\longrightarrow P_0\longrightarrow A\longrightarrow0$$
	with  $\xi\text{-pd} P_i< \infty$ and the resolution $\Mbe$-triangles $\xymatrix{ K_{i+1}\ar[r]&P_i\ar[r]&K_i\ar@{-->}[r]&}$  in $\xi$ for any $ 0\leqslant i \leqslant n-1$ where $K_{n}=K_0=A$, such that $\oplus_{i=1}^nK_i$ is in $\GP$. Then we can get that $K_i$ is in $\GP$  by Proposition \ref{CLOD}, thus each $P_i$ is $\Gproj$ by Lemma \ref{ABC} for any $ 0\leqslant i \leqslant n-1$. By  \citep[Proposition 5.4]{JDP}, We can get that  $\xi\text{-pd}P_i=\Gpd P_i=0$ which implies that  $P_i$ is in $\P$, and by Lemma \ref{FBZH}, we can get  the  $\Mbe$-triangle $ K_{i+1}\longrightarrow P_i\longrightarrow K_i\dashrightarrow$ is $\fbzh$ for all $ 0\leqslant i \leqslant n-1$. It is enough to show $A$ is  $\xi$-$n$-SG-projective.
\end{proof}

\begin{thm}\label{hmy}
	Let $\Mcc$ be a Krull-Schimit category, then for any object $A$ in $\Mcc$, if $A$ is $\xi$-$n$-SG-projective if and only if $\underline{A}$  is $\xi$-$n$-SG-projective, where  $\underline{A}$ is the maximal  direct summands of $A$ without $\xi$-projective  direct summands.
\end{thm}

\begin{proof}
	Let $A=\underline{A}\oplus P$ with $P\in \P$. If $\underline{A}$ is $\xi$-$n$-SG-projective, then $A$ is also $\xi$-$n$-SG-projective by Proposition \ref{PROP2}.
	
	Conversely, assume that $A$  is $\xi$-$n$-SG-projective, then there exists a complete $\P$-exact complex:
	$$0\longrightarrow (A=)\underline{A}\oplus P\longrightarrow P_{n-1}\longrightarrow \cdots\longrightarrow P_0\longrightarrow\underline{A}\oplus P (=A)\longrightarrow0$$
	with $P_i\in\P$ for any $0\leqslant i\leqslant n-1$.
	
	First, for any $0\leqslant i\leqslant n-1$, we have a $\fbzh$ resolution $\Mbe$-triangle $ K_{i+1}\longrightarrow P_i\longrightarrow K_i\dashrightarrow$ in $\xi$ where $K_{n}=K_0=A$. By $\rm{(ET4)}$, there exists a commutative diagram as follows:
	$$\xymatrix{
		P\ar[r]\ar@{=}[d]&A\ar[d]\ar[r]&{\underline{A}}\ar[d]&\\
		P\ar[r]&P_{n-1}\ar[d]\ar[r]&Q_{n-1}\ar[d]&\\
		&K_{n-1}\ar@{=}[r]&K_{n-1}
	}
	$$
	Note that $\xymatrix{A\ar[r]&Q_{n-1}\ar[r]&K_{n-1}\ar@{-->}[r]&}$ is an  $\Mbe$-triangle in $\xi$ since $\xi$ is closed under cobase change. Applying the functor $\Mcc(\P,-)$ to the above diagram, it is easy to see that the $\Mbe$-triangle $\xymatrix{P\ar[r]&P_{n-1}\ar[r]&Q_{n-1}\ar@{-->}[r]&}$ is $\Mcc(\P,-)$-exact by a simply diagram chasing. Therefore,  it is in $\xi$ by Lemma \ref{FBZH1}.
	
	$A$ is $\xi$-$n$-SG-projective, then $K_i$ is $\xi$-$n$-SG-projective by Remark \ref{REM2}(2) for all $0\leqslant i\leqslant n$. So we have $A$ and $K_i$ are in $\Gproj$. It implies that both $\underline{A}$ and $Q_{n-1}$ are also $\Gproj$ by Lemma \ref{ABC} and Lemma \ref{CLOD}. Note that the $\Mbe$-triangle $P\longrightarrow P_{n-1}\longrightarrow Q_{n-1}\dashrightarrow$ is $\fbzh$, because of $Q_{n-1}\in \GP$ and Lemma \ref{FBZH}. So we have following exact sequence in $\mathbf{Ab}$.
	$$0\longrightarrow\Mcc(Q_{n-1},P)\longrightarrow \Mcc(P_{n-1},P)\longrightarrow\Mcc(P,P)\longrightarrow0$$
	This shows the $\Mbe$-triangle $\xymatrix{P\ar[r]&P_{n-1}\ar[r]&Q_{n-1}\ar@{-->}[r]&}$ is split by Lemma \ref{FJ}, i.e. $P_{n-1}\simeq P\oplus Q_{n-1}$. Then one can get that $Q_{n-1}$ is $\xi$-projective. Applying the functor  $\Mcc(-,\P)$ to the above commutative diagram, it is easy to see that the $\Mbe$-triangle $\xymatrix{\underline{A}\ar[r]&Q_{n-1}\ar[r]&K_{n-1}\ar@{-->}[r]&}$ is $\fbzh$ by a diagram chasing.
	
	Next, consider the following commutative diagram by $\rm{(ET4)^{op}}$:
	$$\xymatrix{
		K_1\ar[r]\ar@{=}[d]&Q_0\ar[d]\ar[r]&{\underline{A}}\ar[d]&\\
		K_1\ar[r]&P_{0}\ar[d]\ar[r]&A\ar[d]&\\
		&P\ar@{=}[r]&P\\
	}
	$$
	Then $\xymatrix{K_1\ar[r]&Q_0\ar[r]&{\underline{A}}\ar@{-->}[r]&}$ is an $\Mbe$-triangle in $\xi$ since $\xi$ is closed under base change, and it is $\fbzh$ since $\underline{A}\in\GP$. Applying functor $\Mcc(\P,-)$  to the above commutative diagram, it is easy to see that the triangle $\Mbe$-triangle $Q_0\longrightarrow P_{0}\longrightarrow P\dashrightarrow $ is $\Mcc(\P,-)$-exact by a diagram chasing, so it is in $\xi$ by Lemma \ref{FBZH1}. This shows $P_0\simeq Q_0\oplus P$, thus $Q_0$ is in $\P$ by Remark \ref{REM1}.
	
	So we obtain the following complete $\P$-exact complex:
	$$0\longrightarrow\underline{A}\longrightarrow Q_{n-1}\longrightarrow P_{n-2}\longrightarrow\cdots\longrightarrow P_1\longrightarrow Q_0\longrightarrow\underline{A}\longrightarrow0$$
	That is to say $\underline{A}$ is $\xi$-$n$-SG-projective.
\end{proof}
\begin{cor}\label{ymy}
	If $\xymatrix{A\ar[r]^x&{B}\ar[r]^y&{P}\ar@{-->}[r]^{\delta}&}$ is an $\Mbe$-triangle in $\xi$ with $P\in\P$, then $A\in\nSGP$ if and only if $B\in\nSGP$.
\end{cor}
\begin{prop}\label{tsdj}
	Assume that $A$ and $B$ are $\xi$-projectively equivalent in $\Mcc$. Then, for any $n\geqslant 1$, $A\in \nSGP$ if and only if $B\in\nSGP$.
\end{prop}
\begin{proof}
	Let $A\oplus P\backsimeq B\oplus Q$, where $P$ and $Q$ are belong to $\P$. Similar to the proof of Theorem \ref{hmy},we can get 
\begin{equation*}
	\begin{split}
   A\in \nSGP&\Leftrightarrow A\oplus P\in\nSGP\\
   &\Leftrightarrow B\oplus Q\in\nSGP\\
   &\Leftrightarrow B\in\nSGP
	\end{split}
   \end{equation*}
i.e.  $A\in \nSGP$ if and only if $B\in\nSGP$.
\end{proof}

\begin{cor}
	Let $A$ is $\xi$-$n$-SG-projective and the complex 
$$\mathbf{P}: \cdots\longrightarrow P_{i}\longrightarrow P_{i-1}\longrightarrow\cdots\longrightarrow P_1\longrightarrow P_0\longrightarrow A\longrightarrow0$$
is a $\xi$-projective resolution of $A$,
then any $\xi$-syzygy $K^{\mathbf{P}}_i$ of $A$ is $\xi$-$n$-SG-projective.
\end{cor}
\begin{proof}
Since $A$ is $\xi$-$n$-SG-projective,there exists a complete $\P$-exact complex
$$0\longrightarrow A\longrightarrow Q_{n-1}\longrightarrow Q_{n-2}\longrightarrow\cdots\longrightarrow Q_1\longrightarrow Q_0\longrightarrow A\longrightarrow0$$
where $Q_i\in\P,\ 0\leqslant i\leqslant n$. So the complex
$$\mathbf{Q}: \cdots\longrightarrow Q_1\longrightarrow Q_0\longrightarrow Q_{n-1}\longrightarrow Q_{n-2}\longrightarrow\cdots\longrightarrow Q_1\longrightarrow Q_0\longrightarrow A\longrightarrow0$$
is also a $\xi$-projective resolution of $A$.
By Corollary \ref{Sch}, for any integer $j\geqslant 1$, two $j$th $\xi$-syzygy $K^{\mathbf{P}}_j$ and $K^{\mathbf{Q}}_j$ are $\xi$-projective equivalent. But according to Remark \ref{REM2}(2),we have $K^{\mathbf{Q}}_j$ belongs to $\nSGP$. So $K^{\mathbf{P}}_j$ belongs to $\nSGP$. 
\end{proof}

\begin{cor}\label{FM}
	Let $\xymatrix{A\ar[r]&P\ar[r]&C\ar@{-->}[r]&}$ be an $\Mbe$-triangle in $\xi$ with $P\in\P$. If $C\in\nSGP$, then $A\in\nSGP$.
\end{cor}
\begin{proof}
	Since $C$ is $\xi$-$n$-$\mathcal{G}$projective, there exists a complete $\P$-exact complex:
	$$0\longrightarrow C\longrightarrow P_{n-1}\longrightarrow P_{n-2}\longrightarrow\cdots\longrightarrow P_1\longrightarrow P_0\longrightarrow C\longrightarrow0$$
	with $P_i\in\P$ for any $ 0\leqslant i \leqslant n-1$. Thus for each
	 $0\leqslant i\leqslant n-1$, we have a $\fbzh$ resolution $\Mbe$-triangle
	  $\xymatrix{ K_{i+1}\ar[r]&P_i\ar[r]&K_i\ar@{-->}[r]&}$ in $\xi$, where $K_{n}=K_0=A$.
By Remark \ref{REM2}(2), $K_1\in\nSGP$. Consider the following commutative diagram by Lemma 
\ref{BH}.
\[
	\xymatrix{
		&{K_1}\ar@{=}[r]\ar[d]&{K_1}\ar[d]\\
		A\ar[r]\ar@{=}[d]&B\ar[r]\ar[d]&{P_0}\ar[d]\ar@{-->}[r]&\\
		A\ar[r]&P\ar@{-->}[d]\ar[r]&C\ar@{-->}[r]\ar@{-->}[d]&\\
		&&}
\]
Since $\xi$ is closed under base change, the $\Mbe$-triangles 
\[
	\xymatrix{A\ar[r]&B\ar[r]&{P_0}\ar@{-->}[r]&}\ \text{and}\ \xymatrix{{K_1}\ar[r]&B\ar[r]&{P}\ar@{-->}[r]&}
\] 
are belong to $\xi$, then we can get $B\backsimeq A\oplus P_0\backsimeq K_1\oplus P$ by Remark \ref{REM1}(2).
So $A$ is $\xi$-$n$-$\mathcal{G}$projective by Proposition \ref{tsdj}.
\end{proof}

\begin{cor}
	For any object $A\in\mathcal{C}$ is in $\nSGP$ if and only if there exists an $\Mbe$-triangle $\xymatrix{A\ar[r]&P\ar[r]&G\ar@{-->}[r]&}$ in $\xi$, where $P\in\P$ and $G\in\nSGP$.
\end{cor}

At the end of the section, we study the  relation between the $\Gproj$ and $\xi$-SG-projective.
\begin{thm}\label{dyt}
	Suppose that the  countable direct sums  of $\Gproj$ objects exists and 
	 $\xi$ is closed under the countable direct sums, then $A$ is in $\GP$ if and only if $A$  is a direct summand of some object  in $\SGP$.
\end{thm}
\begin{proof}
	The ``only if'' part is obvious since $\SGP\subseteq \GP$, and $\GP$ is closed under  direct summands.
	
	Conversely, assume that $A$ is $\Gproj$, then there exists a complete $\xi$-projective resolution
	$$
	\mathbf{P}: \xymatrix{ \cdots\ar[r]&P_1\ar[r]^{d_1}&P_{0}\ar[r]^{d_0}&P_{-1}\ar[r]&{\cdots}}
	$$
	in $\Mcc$ such that $P_n$ is projective for each integer $n$ . And for any $P_n$, there exists a $\Mcc(-,\Mcp(\xi))$-exact $\Mbe$-triangle $\xymatrix{ K_{n+1}\ar[r]^{g_n}&P_n\ar[r]^{f_n}&K_n\ar@{-->}[r]^{\delta_n}&}$ in $\xi$ which is the resolution $\mathbb{E}$-triangle of $\mathbf{P}$. Without losing generality, we can assume that $A=K_0$. So we can get a $\fbzh$ $\Mbe$-triangle in $\xi$ as follows
	$$\xymatrix{
		\oplus_{i\in \mathbb{Z}}K_{i+1}\ar[r]&\oplus_{i\in \mathbb{Z}}P_i\ar[r]&\oplus_{i\in \mathbb{Z}}K_i\ar@{-->}[r]&}.
	$$
	Note that $\oplus_{i\in \mathbb{Z}}K_{i+1}\backsimeq\oplus_{i\in \mathbb{Z}}K_i\in\GP$, then $\oplus_{i\in \mathbb{Z}}K_i$ is in $\SGP$. This is enough to show the  ``if" part.
\end{proof}

\begin{lem}
	Suppose that the  countable direct sums  of $\Gproj$ objects exists and 
	 $\xi$ is closed under the countable direct sums, then $\P$, $\GP$ and $\nSGP$ are closed under the countable direct sums.
\end{lem}
\begin{proof}
	It is obvious.
\end{proof}

\begin{defn}
	Let $\mathcal{W}$ be a class of objects in $\Mcc$. We call $\mathcal{W}$ 
	is $\xi$-projectively resolving if

	$(1)$ $\mathcal{W}$ contains all $\xi$-projective objects.

	$(2)$ for any $\Mbe$-triangle $A\rightarrow B\rightarrow C \dashrightarrow $ in
	$\xi$ with $C\in\mathcal{W}$ the conditions $A\in\mathcal{W}$ and $B\in\mathcal{W}$ are equivalent.
\end{defn}

\begin{thm}
	Suppose that the  countable direct sums  of $\Gproj$ objects exists and 
	 $\xi$ is closed under the countable direct sums, then the following are equivalent

$(1)$ $\nSGP$ is closed under extensions;

$(2)$  $\nSGP$ is $\xi$-projectively resolving;

$(3)$ $\GP=\nSGP$;
 
$(4)$ For any $\Mbe$-triangle $G_1\rightarrow G_0\rightarrow A \dashrightarrow $ in
 $\xi$ with $G_0,G_1\in\nSGP$, if $\ext^1(A,P)=0$ for any $P\in\P$, then $A\in\nSGP$.
\end{thm}

\begin{proof}
	$(1)$ $\Rightarrow$ $(2)$ We only need to prove that for any $\Mbe$-triangle $A\rightarrow B\rightarrow C\dashrightarrow $ in
	$\xi$ with $B$ and $C$ in $\nSGP$, then $A\in\nSGP$. Similar to the proof of Corollary \ref{FM}, we have the following commutative diagram 
	\[
		\xymatrix{
			&{K_1}\ar@{=}[r]\ar[d]&{K_1}\ar[d]\\
			A\ar[r]\ar@{=}[d]&D\ar[r]\ar[d]&{P_0}\ar[d]\ar@{-->}[r]&\\
			A\ar[r]&B\ar@{-->}[d]\ar[r]&C\ar@{-->}[r]\ar@{-->}[d]&\\
			&&}
	\]
with some $\Mbe$-triangles in $\xi$. So $D$ is $\xi$-$n$-$\mathcal{G}$projective, since $\nSGP$ is closed under extensions.
Then we have $A\in\nSGP$ by Corollary \ref{ymy}.

$(2)$ $\Rightarrow$ $(3)$ Let $G$ is $\mathcal{G}$projective, then there is a $\mathcal{G}$projective object $H$ such that 
$H\oplus G$ is in $\SGP$ by Theorem \ref{dyt}. Set
\[
L=	H\oplus G\oplus H\oplus G\oplus\cdots\text{,}
\]
then $L\in\SGP$ and so it is in $\nSGP$. We consider the split $\Mbe$-triangle
\[
\xymatrix{G\ar[r]&{G\oplus L}\ar[r]&L\ar@{-->}[r]&}	.
\]
Since $G\oplus L\backsimeq L$, it follows that $G\in\nSGP$.

$(3)$ $\Rightarrow$ $(4)$ It follows from \cite[Lemma 3.6]{JZP}.

$(4)$ $\Rightarrow$ $(1)$ Let  $\xymatrix{A\ar[r]&{B}\ar[r]&C\ar@{-->}[r]&}$ be 
an $\Mbe$-triangle in $\xi$, where $A$ and $C$ belong to $\nSGP$. Then we have the exact sequence
\[
	0=\ext^1(C,P)\rightarrow \ext^1(B,P)\rightarrow \ext^1(A,P)=0
\]
in $\mathbf{Ab}$ for any $P\in\P$ by \cite[Lemma 3.5]{JZP}, so $\ext^1(B,P)=0$.
Since $C$ is $\xi$-$n$-$\mathcal{G}$projective, then there exists an $\Mbe$-triangle 
$\xymatrix{K_1\ar[r]&{P_0}\ar[r]&C\ar@{-->}[r]&}$ in $\xi$ with $K_1\in\nSGP,P_0\in\P$.
Consider the following commutative diagram by Lemma 
\ref{BH}.
\[
	\xymatrix{
		&{K_1}\ar@{=}[r]\ar[d]&{K_1}\ar[d]\\
		A\ar[r]\ar@{=}[d]&D\ar[r]\ar[d]&{P_0}\ar[d]\ar@{-->}[r]&\\
		A\ar[r]&B\ar@{-->}[d]\ar[r]&C\ar@{-->}[r]\ar@{-->}[d]&\\
		&&}
\]
with  $\Mbe$-triangles in $\xi$. Note that $A\in\nSGP$, $D$ is in $\nSGP$ by Corollary \ref{ymy}. And then by hypothesis, we get $B\in\nSGP$.

\end{proof}

\section{$\xi$-$n$-strongly $\mathcal{G}$projective dimension}

\quad~In this section, we introduce the notion of $\xi$-$n$-strongly $\mathcal{G}$projective dimension for any object in $\Mcc$.
\begin{defn}
For any integer $n\geqslant 1$,	the 
\emph{$\xi$-$n$-strongly $\mathcal{G}$projective dimension} $\nGpd A$ of an 
object $A$ is defined inductively. When $A=0$, put $\Gpd A=-1$.  
If $A\in\nSGP$, then define $\Gpd A=0$. Next by induction, for an
 integer $m>0$, put $\nGpd A\leqslant m$ if there exists an 
 $\mathbb{E}$-triangle $K\rightarrow S\rightarrow A\dashrightarrow$ 
 in $\xi$ with $S\in \nSGP$ and $\nGpd K\leqslant n-1$.

 We say $\nGpd A=m$ if $\nGpd A\leqslant m$ and $\nGpd A \nleqslant  m-1$. 
If $\nGpd A\neq m$, for all $m\geqslant0$, we say  $\nGpd A=\infty$.

\end{defn}
\begin{prop}
	There is an inequality $\Gpd A\leqslant \nGpd A\leqslant  \xi\text{-pd}A$, and 
	the equality holds if $  A\in\widehat{\Mcp}(\xi)$ . 
\end{prop}

\begin{proof}
	It is obvious that $\Gpd A\leqslant \nGpd A\leqslant \xi\text{-pd}A$. Assume that $\xi$-pd$A<\infty$, then we have 
	$\Gpd A= \xi$-pd$A$ by \cite[Lemma 5.4]{JDP}, then we can get $\Gpd A=\nGpd A= \xi\text{-pd}A$.

	\begin{cor}
		Let $A$ be an object with $\nGpd A\leqslant m$, then
		
		$(1)$ there exists a $\xi$-exact sequence 
		\[
		0\rightarrow S_m \rightarrow S_{m-1}\rightarrow \cdots	\rightarrow S_1\rightarrow S_0\rightarrow A\rightarrow 0
		\]
		with $S_i\in\nSGP$, $0\leqslant i\leqslant m$ .

		$(2)$ $\ext^{m+i}(A,Q) = 0$ for any $Q \in\ \widehat{\Mcp}(\xi)$ and $i\geqslant 1$.
	\end{cor}

\end{proof}

{\small

}
\end{document}